\documentclass{amsart}

\usepackage{ae,aecompl}
\usepackage{chngcntr}
\usepackage{graphicx}
\usepackage[all,cmtip]{xy}
\usepackage{amsmath, amscd}
\usepackage{amsthm}
\usepackage{amssymb}
\usepackage{amsfonts}
\usepackage{qsymbols}
\usepackage{latexsym}
\usepackage{mathrsfs}
\usepackage{cite}
\usepackage{color}
\usepackage{url}
\usepackage{enumerate}
\usepackage[utf8]{inputenc}
\usepackage[T1]{fontenc}

\usepackage{verbatim}
\usepackage[draft=false, colorlinks=true]{hyperref}

\allowdisplaybreaks

\newcommand {\abs}[1]{\lvert#1\rvert}

\newcommand {\Be}{{B}}

\newcommand {\C}{{\mathbb C}}

\newcommand {\D}{D}

\newcommand {\ud}{\mathrm{d}}
\newcommand {\ue}{e}

\newcommand {\veps}{\varepsilon}

\newcommand {\Ell}{L}
\newcommand {\Ellp}{L^{p}}
\newcommand {\Ellpprime}{L^{p'}}
\newcommand {\Ellq}{L^{q}}

\newcommand {\Ellr}{L^{r}}

\newcommand {\F}{{\mathcal{F}}}
\newcommand {\Hr}{H}

\newcommand {\ui}{i}
\newcommand {\I}{{I}}

\newcommand {\La}{{\mathcal{L}}}
\newcommand {\calL}{{\mathcal{L}}}

\newcommand {\Ma}{{\mathcal{M}}}
\newcommand {\N}{{{\mathbb N}}}

\newcommand {\norm}[1]{\left\|#1\right\|}

\newcommand {\ph}{{\varphi}}

\newcommand {\R}{{\mathbb R}}
\newcommand {\Rd}{{\mathbb{R}^{d}}}

\newcommand {\supp}{{\mathrm{supp}}}
\newcommand {\Se}{\mathrm{S}}

\newcommand {\Sw}{\mathcal{S}}

\newcommand {\w}{{\omega}}
\newcommand {\W}{{\mathrm{W}}}

\newcommand {\Z}{{{\mathbb Z}}}

\newcommand {\vanish}[1]{\relax}

\newcommand{\wh}{\widehat}
\newcommand{\wt}{\widetilde}

\renewcommand{\restriction}{\mathord{\upharpoonright}}
\DeclareMathOperator{\Real}{Re}
\DeclareMathOperator{\Imag}{Im}

\newtheorem{theorem}{Theorem}[section]
\newtheorem{lemma}[theorem]{Lemma}
\newtheorem{proposition}[theorem]{Proposition}
\newtheorem{corollary}[theorem]{Corollary}

\theoremstyle{definition}
\newtheorem{definition}[theorem]{Definition}
\newtheorem{remark}[theorem]{Remark}
\newtheorem{example}[theorem]{Example}

\numberwithin{equation}{section}
\protected\def\ignorethis#1\endignorethis{}
\let\endignorethis\relax

\title{Stability theory for semigroups using $(L^{p},L^{q})$ Fourier multipliers}

\author{Jan Rozendaal}
\address{Mathematical Sciences Institute\\ Australian National University\\ Acton ACT 2601\\ Australia\\ and Institute of Mathematics, Polish Academy of Sciences\\
ul.~\'{S}niadeckich 8\\
00-656 Warsaw\\
Poland}
\email{janrozendaalmath@gmail.com}

\author{Mark Veraar}
\address{Delft Institute of Applied
Mathematics\\
Delft University of Technology\\
P.O.~Box 5031\\
2628 CD Delft\\
The Netherlands}
\email{M.C.Veraar@tudelft.nl}



\thanks{The first author is partially supported by grant DP160100941 of the Australian Research Council. The second author is supported by the VIDI subsidy 639.032.427 of the Netherlands Organisation for Scientific Research (NWO)}

\begin{document}
\begin{abstract}
We study polynomial and exponential stability for $C_{0}$-semigroups using the recently developed theory of operator-valued $(L^{p},L^{q})$ Fourier multipliers. We characterize polynomial decay of orbits of a $C_{0}$-semigroup in terms of the $(L^{p},L^{q})$ Fourier multiplier properties of its resolvent. Using this characterization we derive new polynomial decay rates which depend on the geometry of the underlying space. We do not assume that the semigroup is uniformly bounded, our results depend only on spectral properties of the generator.\\
As a corollary of our work on polynomial stability we reprove and unify various existing results on exponential stability, and we also obtain a new theorem on exponential stability for positive semigroups.
\end{abstract}

\maketitle

\setcounter{tocdepth}{1}

\section{Introduction}\label{sec:introduction}

In this article we study the asymptotic behavior of solutions to the abstract Cauchy problem
\begin{equation}\label{eq:Cauchyintro}
\begin{aligned}
u'(t) + A u(t) &= 0,\quad t\geq 0,
\\  u(0) &= x.
\end{aligned}
\end{equation}
Here $-A$ is the generator of a $C_0$-semigroup $(T(t))_{t\geq 0}$ on a Banach space $X$ and $x\in X$. The unique solution of \eqref{eq:Cauchyintro} with initial data $x$ is given by $u(t) = T(t) x$ for $t\geq 0$.
One of the key difficulties in the asymptotic theory for solutions of \eqref{eq:Cauchyintro} is that the classical Lyapunov stability criterion is in general not valid if $X$ is infinite dimensional. However, asymptotic behavior can be deduced from the associated resolvent operators $R(\lambda,A) = (\lambda-A)^{-1}$ for $\lambda\in \rho(A)$. For example, on a Hilbert space $X$ the Gearhart-Pr\"{u}ss theorem \cite[Theorem 5.2.1]{ArBaHiNe11} states that
$(T(t))_{t\geq0}$ is exponentially stable if and only if $\sigma(A)\subset\C_+$ and $\sup_{\Real(\lambda)<0} \|R(\lambda, A)\|<\infty$. A uniform bound for the resolvent is not sufficient to ensure exponential stability on general Banach spaces, but it was shown in \cite{Hieber01, Latushkin-Shvydkoy01} (see also \cite{Kaashoek-VerduynLunel94,ClLaMoRa00, Latushkin-Rabiger05, Weis97}) that exponential stability can be characterized in terms of $L^{p}$ Fourier multiplier properties of the resolvent.
Outside of Hilbert spaces this multiplier condition is a strictly stronger assumption than uniform boundedness, and in applications it can be difficult to verify. On the other hand, cf.~\cite{vNeStWe95, Weis95, Weis97, vanNeerven09}, uniform bounds for the resolvent do imply exponential stability for orbits in fractional domains, with the fractional domain parameter depending on the geometry of the underlying space. At the moment it is not fully understood how the characterization of exponential stability using Fourier multipliers is related to such concrete decay results.

In a separate development, over the past decade much attention has been paid to polynomial decay of semigroup orbits. The work of Lebeau \cite{Lebeau96,Lebeau-Robbiano97} and Burq \cite{Burq98} on energy decay for damped wave equations raised the question what the precise relation is between growth rates for the resolvent and decay rates for the semigroup. More precisely, if one has $\sigma(A)\subset\C_{+}$ in \eqref{eq:Cauchyintro} but $\|R(i\xi,A)\|\to\infty$ as $|\xi|\to\infty$, then $(T(t))_{t\geq0}$ is not exponentially stable and one typically encounters other asymptotic behavior. Since a uniform rate of decay for all solutions to \eqref{eq:Cauchyintro} implies exponential stability of the semigroup, one can expect uniform asymptotic behavior only for orbits in suitable subspaces such as fractional domains, and in general the smoothness parameter of the fractional domain influences the decay behavior. In \cite{BaEnPrSc06} B{\'a}tkai, Engel, Pr{\"u}ss and Schnaubelt proved that for uniformly bounded semigroups a polynomial growth rate of the resolvent implies a specific polynomial decay rate for classical solutions of \eqref{eq:Cauchyintro} and vica versa, and they showed that this correspondence is optimal up to an arbitrarily small polynomial loss. In \cite{Batty-Duyckaerts08} Batty and Duyckaerts extended this correspondence to the setting of arbitrary resolvent growth and they reduced the loss to a logarithmic scale. Then Borichev and Tomilov proved in \cite{Borichev-Tomilov10} that this logarithmic loss is sharp on general Banach spaces, but that it can be removed on Hilbert spaces in the case of polynomial resolvent growth. In particular, on Hilbert spaces this yields a characterization of polynomial stability in terms of the growth of the resolvent. This result has been applied extensively in the study of partial differential equations (see e.g.\ \cite{AmFeNi14, Anantharaman-Leautaud14,BaPaSe16,CaCaTe17,Hao-Liu13,Leautaud-Lerner17,Liu-Rao05, SaAlMu12} and references therein) and has been extended in \cite{BaChTo16,Martinez11, Chill-Seifert16, Seifert15a,Stahn17b,Stahn17e,RoSeSt17} to finer scales of resolvent growth and semigroup decay.

Although much work has gone into determining the relation between resolvent growth and polynomial rates of decay, it is not clear how such asymptotic behavior relates to the Fourier analytic properties of the resolvent which characterize exponential stability. Furthermore, the currently available literature on polynomial decay deals almost exclusively with uniformly bounded semigroups. To the best of our knowledge, the only previously known result concerning polynomial decay for general semigroups is \cite[Proposition 3.4]{BaEnPrSc06}. There are many natural classes of examples where the generator has spectral properties as above but the semigroup is not uniformly bounded, or where it is unknown whether the semigroup is bounded. Typical examples of this phenomenon can be found in Section \ref{subsec:comparison, problems and examples} and include semigroups whose generator is an operator matrix or a multiplication operator on a Sobolev space. In turn, such operators can be found in disguise in concrete partial differential equations. One example is the standard wave equation with periodic boundary conditions; here uniform boundedness fails. Other examples can be found in \cite{Paunonen14} for certain classes of perturbed wave equations and in \cite{Sklyar-Polak17} for delay equations. For infinite systems of equations the uniform boundedness condition leads to additional assumptions on the coefficients in \cite{Paunonen-Seifert17}.

In this article we deal with the problems outlined above in three ways. First, we characterize polynomial stability on general Banach spaces in terms of Fourier multiplier properties of powers of the resolvent, in Theorem \ref{thm:abstract polynomial stability}. In doing so we extend the Fourier analytic characterization of exponential stability to this more refined setting. Then, using the theory of operator-valued $(L^{p},L^{q})$ Fourier multipliers which was developed in \cite{Rozendaal-Veraar17a,Rozendaal-Veraar17b} with applications to stability theory in mind, we derive concrete polynomial decay rates from this characterization. These results involve only growth bounds for the resolvent and are new even on Hilbert spaces. In particular, the following theorem can be found in the main text as Corollary \ref{cor:polynomial decay Hilbert space}.

\begin{theorem}\label{thm:polynomial decay HilbertIntro}
Let $-A$ be the generator of a $C_{0}$-semigroup $(T(t))_{t\geq0}$ on a Hilbert space $X$ such that $\sigma(A)\subset\C_{+}$ and $\|R(\lambda,A)\|\leq C(1+|\lambda|)^{\beta}$ for some $\beta>0$, $C\geq 0$ and all $\lambda\in\C$ with $\Real(\lambda)\leq 0$. Then for each $\tau\geq \beta$ there exists a $C_{\tau}\geq0$ such that
\begin{equation}\label{eq:polynomial decay Hilbert space mainintro}
\|T(t)A^{-\tau}\|\leq C_{\tau}t^{1-\tau/\beta}\qquad (t\in[1,\infty)).
\end{equation}
\end{theorem}

Note that we do not assume that the semigroup is uniformly bounded. In fact, we show that one can derive polynomial decay behavior for initial values in suitable fractional domains given only spectral properties of the generator. In particular, by setting $\tau=\beta$ in Theorem \ref{thm:polynomial decay HilbertIntro} one obtains uniform boundedness of sufficiently smooth solutions. For uniformly bounded semigroups the parameter $1-\tau/\beta$ in \eqref{eq:polynomial decay Hilbert space mainintro} can be replaced by $-\tau/\beta$, as was shown in \cite{Borichev-Tomilov10}, but in Example \ref{ex:optimalitybeta} we prove that $1-\tau/\beta$ is optimal for general semigroups if $\tau=\beta$. Our main theorems allow for $A$ to have a singularity at zero, or even singularities at both zero and infinity. We also obtain versions of Theorem \ref{thm:polynomial decay HilbertIntro} on other Banach spaces; the decay rate in \eqref{eq:polynomial decay Hilbert space mainintro} then depends on the geometry of the underlying space.

Finally, as a direct corollary of our results on polynomial stability we recover in a unified manner various results on exponential stability from \cite{Hieber01,Latushkin-Shvydkoy01,vNeStWe95, Weis95, Weis97, vanNeerven09}. We also obtain a new stability result for positive semigroups, Theorem \ref{thm:exponential stability convex}.

To prove our main results we rely on the theory of operator-valued Fourier multipliers from $L^p(\R;X)$ to $L^q(\R;Y)$, for $X$ and $Y$ Banach spaces. A Fourier multiplier characterization of exponential stability for general $p\in[1,\infty)$ and $q\in[p,\infty]$ was known from \cite{Latushkin-Shvydkoy01}, but so far only the case where $p=q$ has been used (see \cite{BaFaSh03,Hieber99,Hieber01,Latushkin-Rabiger05, Latushkin-Shvydkoy01,Weis97}). Although in this setting very powerful multiplier theorems are available, see for example Weis' version of the Mikhlin multiplier theorem in \cite{Weis2001} and \cite{DeHiPr03,HyNeVeWe16,Kunstmann-Weis04}, the assumptions of these theorems are in general too restrictive for applications to stability theory. Indeed, multiplier theorems on $L^{p}(\R;X)$ typically require both a geometric assumption on $X$, namely the UMD condition which excludes spaces of interest such as $X=L^{1}$, as well as smoothness of the multiplier and comparatively fast decay at infinity of its derivative. The latter assumption in particular is not satisfied in most applications to stability theory.

In this article we argue that for the study of asymptotic behavior it is more natural to consider general $p\in[1,\infty)$ and $q\in[p,\infty]$. It was observed in \cite{Rozendaal-Veraar17a, Rozendaal-Veraar17b} that one can derive boundedness of Fourier multipliers from $L^{p}(\R;X)$ to $L^{q}(\R;Y)$ for $p< q$ under different geometric assumptions on $X$ and $Y$ than in the case where $p=q$, and assuming decay of the multiplier at infinity but no smoothness. In fact, the parameters $p$ and $q$ depend on the geometry of $X$, and the amount of decay which is required at infinity is proportional to $\frac{1}{p}-\frac{1}{q}$. Moreover, in Section \ref{subsec:resolvent estimates} we prove that growth of the resolvent on $X$ corresponds to uniform boundedness, and in fact even decay, of the resolvent from suitable fractional domain and range spaces to $X$. Then one can determine for which fractional domain and range parameters the conditions of the $(L^{p},L^{q})$ multiplier theorems are satisfied for (powers of) the resolvent, and the Fourier multiplier characterizations of stability in Theorems \ref{thm:abstract polynomial stability} and \ref{thm:abstract stability result} yield the corresponding asymptotic behavior. We emphasize that, although we use Fourier multiplier techniques for the proofs, our main theorems on concrete decay rates involve only growth bounds on the resolvent.

This article is organized as follows. In Section \ref{sec:preliminaries} we present some basics on Banach space geometry, Fourier multipliers and sectorial operators. In Section \ref{sec:resolvents and multipliers} we deduce multiplier properties of the resolvent and we prove Proposition \ref{prop:app-equivalence growth bounds} and Corollary \ref{cor:app-growth and fractional domains}. These are fundamental in later sections for relating resolvent growth on $X$ to boundedness and decay from fractional domain and range spaces to $X$. In Section \ref{sec:polynomial stability} we study polynomial decay of semigroups. We characterize polynomial stability using Fourier multipliers, and from this characterization we deduce concrete polynomial decay rates which depend on the geometry of the underlying space. In Section \ref{sec:exponential stability} we derive from these results various corollaries on exponential decay. We also prove a characterization of exponential stability using multipliers on Besov spaces, which in turn is used to obtain a new stability result for positive semigroups. An appendix contains estimates for contour integrals and exponential functions.

\subsection{Notation}\label{subsec:notation}

\medskip

The set of natural numbers is $\N=\{1,2,\ldots\}$, and $\N_{0}:=\N\cup\{0\}$. We denote by $\C_{+}=\{\lambda\in\C\mid \Real(\lambda)>0\}$ and $\C_{-}=-\C_{+}$ the open complex right and left half-planes.

Nonzero Banach spaces over the complex numbers are denoted by $X$ and $Y$. The space of bounded linear operators from $X$ to $Y$ is $\La(X,Y)$, and $\La(X):=\La(X,X)$. The identity operator on $X$ is denoted by $\I_{X}$, and we usually write $\lambda$ for $\lambda\I_{X}$ when $\lambda\in\C$. The domain of a closed operator $A$ on $X$ is $\D(A)$, a Banach space with the norm
\begin{align*}
\norm{x}_{\D(A)}:=\norm{x}_{X}+\norm{Ax}_{X}\qquad (x\in \D(A)).
\end{align*}
For an injective closed operator $A$ we identify the range $\text{ran}(A)$ of $A$ with the Banach space $D(A^{-1})$. The spectrum of $A$ is $\sigma(A)$ and the resolvent set is $\rho(A)=\C\setminus \sigma(A)$. We write $R(\lambda,A)=(\lambda -A)^{-1}$ for the resolvent operator of $A$ at $\lambda\in\rho(A)$.

For $p\in[1,\infty]$ and $\Omega$ a measure space, $\Ellp(\Omega;X)$ is the Bochner space of equivalence classes of strongly measurable, $p$-integrable, $X$-valued functions on $\Omega$.
The H\"{o}lder conjugate of $p\in[1,\infty]$ is denoted by $p'$ and is defined by $1=\frac{1}{p}+\frac{1}{p'}$.

The class of $X$-valued Schwartz functions on $\R$ is denoted by $\Sw(\R;X)$, and the space of $X$-valued tempered distributions by $\Sw'(\R;X)$. The Fourier transform of $f\in\Sw'(\R;X)$ is denoted by $\F f$ or $\widehat{f}$. If $f\in\Ell^{1}(\R;X)$ then
\begin{align*}
\F f(\xi)=\int_{\R}\ue^{-\ui \xi t}f(t)\,\ud t\qquad (\xi\in\R).
\end{align*}

We use the convention that $\tfrac{1}{0}=\infty$ and $\tfrac{0}{0}=\infty$.

For sets $S$ and $Z$ we occasionally denote a function $f:S\to Z$ of a variable $s$ simply by $f=f(s)$. We use the notation $f(s)\lesssim g(s)$ for functions $f,g:S\to\R$ to indicate that $f(s)\leq Cg(s)$ for all $s\in S$ and a constant $C\geq0$ independent of $s$, and similarly for $f(s)\gtrsim g(s)$. We write $f(s)\eqsim g(s)$ if $g(s)\lesssim f(s)\lesssim g(s)$ holds.

\section{Preliminaries}\label{sec:preliminaries}

\subsection{Banach space geometry}\label{subsec:banach space geometry}

Here we collect some background on Banach space geometry which is used for our results on non-Hilbertian Banach spaces.

A Banach space $X$ has \emph{Fourier type} $p\in[1,2]$ if the Fourier transform $\F$ is bounded from $\Ellp(\R;X)$ to $\Ellpprime(\R;X)$. We then set $\F_{p,X}:=\|\F\|_{\La(\Ell^{p}(\R;X),\Ell^{p'}(\R;X))}$. To make our multiplier theorems more transparent, we say that $X$ has \emph{Fourier cotype} $q\in[2,\infty]$ if $X$ has Fourier type $q'$. Each Banach space has Fourier type $1$, and $X$ has Fourier type $2$ if and only if $X$ is isomorphic to a Hilbert space. For $r\in[1,\infty]$ and $\Omega$ a measure space, $\Ellr(\Omega)$ has Fourier type $\min(r,r')$. For more on Fourier type see \cite{Pietsch-Wenzel98, HyNeVeWe16}.

A (real) Rademacher variable is a random variable $r:\Omega\to\{-1,1\}$ on a probability space $(\Omega,\mathbb{P})$ such that $\mathbb{P}(r=-1)=\mathbb{P}(r=1)=\frac{1}{2}$. A \emph{Rademacher sequence} is a sequence $(r_{k})_{k\geq 1}$ of independent Rademacher variables on some probability space.

Let $(r_{k})_{k\geq 1}$ be a Rademacher sequence on a probability space $(\Omega,\mathbb{P})$. A Banach space $X$ has \emph{type} $p\in[1,2]$ if there exists a constant $C\geq0$ such that for all $n\in\N$ and all $x_{1},\ldots, x_{n}\in X$ one has
\[
\Big(\mathbb{E}\Big\|\sum_{k=1}^{n}r_{k}x_{k}\Big\|^{2}\Big)^{1/2}\leq C\Big(\sum_{k=1}^{n}\|x_{k}\|^{p}\Big)^{1/p}\!.
\]
Also, $X$ has \emph{cotype} $q\in[2,\infty]$ if there exists a constant $C\geq0$ such that for all $n\in\N$ and all $x_{1},\ldots, x_{n}\in X$ one has
\[
\Big(\sum_{k=1}^{n}\|x_{k}\|^{q}\Big)^{1/q}\leq C\Big(\mathbb{E}\Big\|\sum_{k=1}^{n}r_{k}x_{k}\Big\|^{2}\Big)^{1/2}\!,
\]
with the obvious modification for $q=\infty$. We say that $X$ has \emph{nontrivial type} if $X$ has type $p\in(1,2]$, and \emph{finite cotype} if $X$ has cotype $q\in[2,\infty)$. Each Banach space has type $p=1$ and cotype $q=\infty$, and $X$ has type $p=2$ and cotype $q=2$ if and only if $X$ is isomorphic to a Hilbert space, by Kwapie\'{n}'s theorem \cite{Kwapien72}. For $r\in[1,\infty)$ and $\Omega$ a measure space, $\Ellr(\Omega)$ has type $\min(r,2)$ and cotype $\max(r,2)$. For more on type and cotype see \cite{DiJaTo95,HyNeVeWe2}.

Let $X$ be a Banach lattice and $p,q\in[1,\infty]$. We say that $X$ is \emph{$p$-convex} if there exists a constant $C\geq0$ such that for all $n\in\N$ and all $x_{1},\ldots, x_{n}\in X$ one has
\[
\Big\|\Big(\sum_{k=1}^{n}\abs{x_{k}}^{p}\Big)^{1/p}\Big\|_{X}\leq C\Big(\sum_{k=1}^{n}\|x_{k}\|_{X}^{p}\Big)^{1/p}\!,
\]
with the obvious modification for $p=\infty$. We say that $X$ is \emph{$q$-concave} if there exists a constant $C\geq0$ such that for all $n\in\N$ and all $x_{1},\ldots, x_{n}\in X$ one has
\[
\Big(\sum_{k=1}^{n}\|x_{k}\|_{X}^{q}\Big)^{1/q}\leq C\Big\|\Big(\sum_{k=1}^{n}\abs{x_{k}}^{q}\Big)^{1/q}\Big\|_{X},
\]
with the obvious modification for $q=\infty$. Each Banach lattice $X$ is $1$-convex and $\infty$-concave.
For $r\in[1,\infty]$ and $\Omega$ a measure space, $\Ellr(\Omega)$ is $r$-convex and $r$-concave. For more on $p$-convexity and $q$-concavity we refer the reader to \cite{Lindenstrauss-Tzafriri79,GaToKa96}.

Let $X$ and $Y$ be Banach spaces and $\mathcal{T}\subseteq \La(X,Y)$. We say that $\mathcal{T}$ is \emph{$R$-bounded}
if there exists a constant $C\geq0$ such that for all $n\in\N$, $T_{1},\ldots, T_{n}\in\mathcal{T}$ and $x_{1},\ldots, x_{n}\in X$ one has
\begin{equation}\label{eq:R-boundedness}
\Big(\mathbb{E}\Big\|\sum_{k=1}^{n}r_{k}T_{k}x_{k}\Big\|_{Y}^{2}\Big)^{1/2}\leq
C\Big(\mathbb{E}\Big\|\sum_{k=1}^{n}r_{k}x_{k}\Big\|_{X}^{2}\Big)^{1/2}\!.
\end{equation}
The smallest such $C$ is the \emph{$R$-bound} of $\mathcal{T}$ and is denoted by $R(\mathcal{T})$.
If we want to specify the underlying spaces $X$ and $Y$ then we write $R_{X,Y}(\mathcal{T})$ for the $R$-bound of $\mathcal{T}$, and we write $R_{X}(\mathcal{T})=R_{X,Y}(\mathcal{T})$ if $X=Y$. Every $R$-bounded collection is uniformly bounded with supremum bound less than or equal to its $R$-bound, and the converse holds if and only if $X$ has cotype $2$ and $Y$ has type $2$.
For $\lambda\in\C$ and an $R$-bounded collection $\mathcal{T}\subseteq\La(X,Y)$, the closed absolutely convex hull $\overline{\text{aco}}(\lambda\mathcal{T})\subseteq\La(X,Y)$ of $\lambda\mathcal{T}=\{\lambda T\mid T\in\mathcal{T}\}$ is $R$-bounded, and
\begin{equation}\label{eq:Kahane contraction}
R_{X,Y}(\overline{\text{aco}}(\lambda\mathcal{T}))\leq 2\abs{\lambda}R_{X,Y}(\mathcal{T}).
\end{equation}
In particular, $L^{1}$-averages of $R$-bounded collections are again $R$-bounded, a fact which will be used frequently. For more on $R$-boundedness see \cite{Kunstmann-Weis04, HyNeVeWe2, vanNeerven10}.

The following lemma is used in the proof of Corollary \ref{cor:R-boundedness for free}. It can also be deduced from a corresponding statement in \cite[Theorem 5.1]{Hytonen-Veraar09} for the Besov space $B^{1/r}_{r,1}(\R;\calL(X,Y))$. Here we give a more direct proof. For $r\in[1,\infty]$ and $E$ a Banach space we denote by $W^{1,r}(\R;E)$ the Sobolev space of weakly differentiable $f:\R\to E$ such that $f,f'\in L^{r}(\R;E)$, with $\|f\|_{W^{1,r}(E)}:=\|f\|_{L^{r}(\R;E)}+\|f'\|_{L^{r}(\R;E)}$.

\begin{lemma}\label{lem:rbddsuff}
Let $X$ be a Banach space with cotype $q\in[2,\infty)$ and $Y$ a Banach space with type $p\in[1,2]$, and let $r\in[1,\infty]$ be such that $\frac{1}{r} = \frac1p-\frac1q$. Then there exists a constant $C\in[0,\infty)$ such that for all $f\in W^{1,r}(\R;\calL(X,Y))$ the set $\{f(t)\mid t\in \R\}\subseteq\La(X,Y)$ is $R$-bounded, with
\[
R(\{f(t)\mid t\in \R\})\leq C \|f\|_{W^{1,r}(\R;\calL(X,Y))}.
\]
\end{lemma}
\begin{proof}
Let $f\in W^{1,r}(\R;\La(X,Y))$ and for $j\in\Z$ set $I_j := [j, j+1)$ and $\mathcal{T}_j := \{f(t)\mid t\in I_j\}$. Then \cite[Example 2.18]{Kunstmann-Weis04} and H\"{o}lder's inequality imply
\[
R(\mathcal{T}_j) \lesssim \|f\|_{W^{1,1}(I_j;\calL(X,Y))}\lesssim\|f\|_{W^{1,r}(I_j;\calL(X,Y))}
\]
for all $j\in \Z$. Now \cite[Theorem 3.1]{vanGaans06} (see also \cite[Proposition 9.1.10]{HyNeVeWe2}) shows that $\{f(t)\mid t\in \R\} = \bigcup_{j\in \Z} \mathcal{T}_j$ is $R$-bounded, with
\begin{align*}
R(\{f(t)\mid t\in \R\}) &\lesssim \|(R\big(\mathcal{T}_j)\big)_{j}\|_{\ell^{r}(\Z)}
\lesssim \|\big(\|f\|_{W^{1,r}(I_{j};\La(X,Y))}\big)_{j}\|_{\ell^{r}(\Z)}\\
&\lesssim \|f\|_{W^{1,r}(\R;\calL(X,Y))}.\qedhere
\end{align*}
\end{proof}

By replacing the Rademacher random variables in \eqref{eq:R-boundedness} by Gaussian variables, one obtains the definition of a \emph{$\gamma$-bounded} collection $\mathcal{T}\subseteq\La(X,Y)$. Each $R$-bounded collection is $\gamma$-bounded, and the converse holds if and only if $X$ has finite cotype (see \cite[Theorem 1.1]{KwVeWe14}). We choose to work with $R$-boundedness in this article, both because the notion of $R$-boundedness is more established and because those stability theorems in this article which use $R$-boundedness are only of interest on spaces with finite cotype.

\subsection{Fourier multiplier theorems\label{subsec:Fourierm}}

To properly define Fourier multipliers for symbols with a singularity at zero, we briefly introduce the class of vector-valued homogeneous distributions. For more on these distributions see \cite{Rozendaal-Veraar17b}. For $X$ a Banach space let
\[
\dot{\Sw}(\R;X):=\{f\in\Sw(\R;X)\mid \wh{f}^{\,(k)}(0)=0\text{ for all }k\in\N_{0}\},
\]
endowed with the subspace topology, and let $\dot{\Sw}'(\R;X)$ be the space of continuous linear mappings from $\dot{\Sw}(\R;\C)$ to $X$. Then $\dot{\Sw}(\R;X)$ is dense in $L^{p}(\R;X)$ for all $p\in[1,\infty)$, and $L^{p}(\R;X)$ can be naturally identified with a subspace of $\dot{\Sw}'(\R;X)$ for all $p\in[1,\infty]$.

Let $X$ and $Y$ be Banach spaces. A function $m:\R\setminus\{0\}\to\La(X,Y)$ is \emph{$X$-strongly measurable} if $\xi\mapsto m(\xi)x$ is a strongly measurable $Y$-valued map for each $x\in X$. We say that $m$ is \emph{of moderate growth} if there exist $\alpha\in[0,\infty)$ and $g\in \Ell^1(\R)$ such that
\begin{align*}
|\xi|^{\alpha}(1+\abs{\xi})^{-2\alpha} \|m(\xi)\|_{\La(X,Y)} \leq g(\xi) \qquad (\xi\in\R).
\end{align*}
Let $m:\R\setminus\{0\}\to\La(X,Y)$ be an $X$-strongly measurable map of moderate growth. Then $T_{m}:\dot{\Sw}(\R;X)\to\dot{\Sw}'(\R;Y)$,
\begin{equation}\label{eq:definition multiplier}
T_{m}(f):=\F^{-1}(m\cdot\widehat{f}\,)\qquad (f\in\dot{\Sw}(\R;X)),
\end{equation}
is the \emph{Fourier multiplier operator} associated with $m$. One calls $m$ the \emph{symbol} of $T_{m}$, and we identify symbols which are equal almost everywhere. If $\|m(\cdot)\|_{\La(X,Y)}\in L^{1}_{\text{loc}}(\R)$ then \eqref{eq:definition multiplier} extends to all $f\in \Sw(\R;X)$ and defines an operator $T_{m}:\Sw(\R;X)\to \Sw'(\R;X)$.

For $p\in[1,\infty)$ and $q\in[1,\infty]$ we let $\Ma_{p,q}(\R;\La(X,Y))$ be the set of all $X$-strongly measurable $m:\R\setminus\{0\}\to\La(X,Y)$ of moderate growth such that $T_{m}\in\La(L^{p}(\R;X),L^{q}(\R;Y))$, and
\[
\|m\|_{\Ma_{p,q}(\R;\La(X,Y))}:=\|T_{m}\|_{\La(L^{p}(\R;X),L^{q}(\R;Y))}.
\]
We write $\|\cdot\|_{\Ma_{p,q}}=\|\cdot\|_{\Ma_{p,q}(\R;\La(X,Y))}$ when the spaces $X$ and $Y$ are clear from the context.

We now recall several $(L^{p},L^{q})$ Fourier multiplier results from our earlier work. The first is \cite[Proposition 3.9]{Rozendaal-Veraar17a}.

\begin{proposition}\label{prop:Lp-Lq multipliers Fourier type}
Let $X$ be a Banach space with Fourier type $p\in[1,2]$ and $Y$ a Banach space with Fourier cotype $q \in[2,\infty]$, and let $r\in [1, \infty]$ be such that $\frac{1}{r} = \frac{1}{p} - \frac{1}{q}$. Let $m:\R\setminus\{0\}\to\La(X,Y)$ be an $X$-strongly measurable map such that $\norm{m(\cdot)}_{\La(X,Y)}\in\Ellr(\R)$. Then $m\in \Ma_{p,q}(\R;\La(X,Y))$ and
\begin{equation}\label{eq:Fourier type multipliers}
\|m\|_{\Ma_{p,q}(\R;\La(X,Y))} \leq \frac{1}{2\pi}\F_{p,X}\F_{q',Y}\norm{\|m(\cdot)\|_{\La(X,Y)}}_{\Ellr(\R)}.
\end{equation}
\end{proposition}

Our next result follows from \cite[Theorem 4.6 and Remark 4.8]{Rozendaal-Veraar17b} and \cite[Theorem 3.21 and Remark 3.22]{Rozendaal-Veraar17a}.

\begin{proposition}\label{prop:Lp-Lq multipliers type}
Let $X$ be a Banach space with type $p\in[1,2]$ and $Y$ a Banach space with cotype $q\in[2,\infty]$, and let $r\in [1, \infty]$ be such that $\frac{1}{r}>\frac{1}{p} - \frac{1}{q}$. Then there exists a constant $C\in[0,\infty)$ such that the following holds. Let $m:\R\to\La(X,Y)$ be an $X$-strongly measurable map such that $\{(1+\abs{\xi})^{r}m(\xi)\mid \xi\in \R\}\subseteq\La(X,Y)$ is $R$-bounded. Then $m\in \Ma_{p,q}(\R;\La(X,Y))$ and
\begin{equation}\label{eq:multiplier type/cotype}
\|m\|_{\Ma_{p,q}(\R;\La(X,Y))}\leq C R_{X,Y}(\{(1+\abs{\xi})^{r}m(\xi)\mid \xi\in \R\}).
\end{equation}
Moreover, if $X$ is a complemented subspace of a $p$-convex Banach lattice with finite cotype and if $Y$ is a Banach space continuously embedded in a $q$-concave Banach lattice for $q\in[1,\infty)$, then \eqref{eq:multiplier type/cotype} also holds if $\frac{1}{r}=\frac{1}{p}-\frac{1}{q}$.
\end{proposition}

For $s\in\R$ and $p\in[1,\infty]$, the \emph{inhomogeneous Bessel potential space} $\Hr^{s}_{p}(\R;X)$ consists of all $f\in\Sw'(\R;X)$ such that $T_{m_{s}}(f)\in\Ellp(\R;X)$, where $m_{s}(\xi):=(1+|\xi|^{2})^{s/2}$ for $\xi\in\R$. It is a Banach space endowed with the norm
\begin{align*}
\|f\|_{\Hr^{s}_{p}(\R;X)}:=\|T_{m_{s}}(f)\|_{\Ellp(\Rd;X)}\quad (f\in\Hr^{s}_{p}(\R;X)).
\end{align*}
Moreover, $\dot{\Sw}(\Rd;X)\subseteq\Hr^{s}_{p}(\Rd;X)$ is densely embedded for $p<\infty$.

The following proposition is proved in the same way as the corresponding homogeneous version in \cite[Theorem 3.24]{Rozendaal-Veraar17a}. We note that one can often avoid condition \eqref{KcondL1} by using approximation arguments.

\begin{proposition}\label{prop:positive kernel Lp-Lq multipliers}
Let $p\in[1,\infty)$ and $q\in[p,\infty)$. Let $X$ be a $p$-convex Banach lattice with finite cotype and let $Y$ be a $q$-concave Banach lattice, and let $r\in (1, \infty]$ be such that $\frac1r = \frac1p-\frac1q$. Then there exists a constant $C\in[0,\infty)$ such that the following holds. Let $m:\R\to \La(X,Y)$ be such that there exists a $K:\R\to \calL(X,Y)$ satisfying the following conditions:
\begin{enumerate}[(1)]
\item\label{Kcondpos} $K(t)\in\La(X,Y)$ is a positive operator for all $t\in \R$;
\item\label{KcondL1} $K(\cdot)x\in\Ell^{1}(\R;Y)$ for all $x\in X$;
\item\label{KcondF} $\F(K(\cdot)x)(\xi)=m(\xi)x$ for all $x\in X$ and $\xi\in\R$.
\end{enumerate}
Then $T_m:H^{1/r}_p(\R;X)\to L^{q}(\R;Y)$ is bounded and
\begin{align*}
\|T_m\|_{\La(H^{1/r}_p(\R;X),L^{q}(\R;Y))} \leq C\|m(0)\|_{\calL(X,Y)} \leq C\sup_{\xi\in\R}\|m(\xi)\|_{\La(X,Y)}.
\end{align*}
\end{proposition}

\subsection{Sectorial operators}\label{subsec:sectorial operators}

For a $C_{0}$-semigroup $(T(t))_{t\geq 0}\subseteq\La(X)$ on a Banach space $X$ we let
\begin{align*}
\w_{0}(T):=\inf\{\w\in\R\mid \exists M\in[0,\infty): \|T(t)\|_{\La(X)}\leq M\ue^{\w t}\text{ for all }t\in[0,\infty)\}.
\end{align*}
For $\ph\in(0,\pi)$ set
\[
S_{\ph}:=\{z\in \C\setminus\{0\}\mid \abs{\arg(z)}<\ph\},
\]
and let $S_{0}:=(0,\infty)$. Recall that an operator $A$ on a Banach space $X$ is \emph{sectorial of angle} $\ph\in[0,\pi)$ if $\sigma(A)\subseteq \overline{S_{\ph}}$ and if $\sup\{\|\lambda R(\lambda,A)\|_{\La(X)}\mid \lambda\in \C\setminus\overline{S_{\theta}}\}<\infty$ for all $\theta\in(\ph,\pi)$. Then we write $A\in\text{Sect}(\ph,X)$ and we let $\w_{A}:=\min\{\ph\in[0,\pi)\mid A\in \text{Sect}(\ph,X)\}$. An operator $A$ such that
\begin{equation}\label{eq:sectoriality constant}
M(A):=\sup\{\|\lambda(\lambda+A)^{-1}\|_{\La(X)}\mid \lambda\in(0,\infty)\}<\infty
\end{equation}
is sectorial of angle $\ph=\pi-\arcsin(1/M(A))$.

For a sectorial operator $A$ on a Banach space $X$ one has $N(A)\cap \overline{\text{Ran}(A)}=\{0\}$
and, if $X$ is reflexive, $X=N(A)\oplus \overline{\text{Ran}(A)}$. If $-A$ generates a $C_{0}$-semigroup $(T(t))_{t\geq 0}\subseteq\La(X)$ then $T(t)x=x$ for all $x\in N(A)$ and $t\geq0$.
Moreover, the restriction of $(T(t))_{t\geq 0}$ to $\overline{\text{Ran}(A)}$ is generated by the part of $A$ in $\overline{\text{Ran}(A)}$, which is injective.
Hence for the purposes of stability theory it is natural to assume that $A$ is injective, and we will do so frequently.

For the definition and various properties of fractional powers of sectorial operators we refer to \cite{Haase06a,Martinez_Carracedo-Sanz_Alix01}. We shall use in particular that, for $\ph\in[0,\pi)$, $A\in \text{Sect}(\ph,X)$ and $\alpha,\beta,\eta\in(0,\infty)$, one has
\begin{equation}\label{eq:fractional power}
A^{\alpha}(\eta+A)^{-\alpha-\beta}=\frac{1}{2\pi\ui}\int_{\partial \Se_{\theta}}\frac{z^{\alpha}}{(\eta+z)^{\alpha+\beta}}R(z,A)\ud z.
\end{equation}
Here $\partial \Se_{\theta}$ is the positively oriented boundary of $\Se_{\theta}$ for $\theta\in (\ph,\pi)$. Note that $A^{\alpha}$ is injective for $A$ injective, and if $A$ is invertible then one may let $\alpha=0$ in \eqref{eq:fractional power}.

For $A$ a sectorial operator and $\alpha,\beta\in[0,\infty)$ we set $\Phi^{\alpha}_{\beta}(A):=A^{\alpha}(1+A)^{-\alpha-\beta}\in\La(X)$. We will frequently use that $\Phi^{\alpha}_{0}(A)=(A(1+A)^{-1})^{\alpha}$ and that
\begin{equation}\label{eq:functional calculus}
\Phi^{\alpha_{1}}_{\beta_{1}}(A)\Phi^{\alpha_{2}}_{\beta_{2}}(A)=\Phi^{\alpha_{1}+\alpha_{2}}_{\beta_{1}+\beta_{2}}(A)
\end{equation}
for $\alpha_{1},\alpha_{2},\beta_{1},\beta_{2}\in[0,\infty)$, by \cite[Proposition 3.1.1]{Haase06a}. Let $X^{\alpha}_{\beta}:=\text{Ran}(\Phi^{\alpha}_{\beta}(A))$, $X^{\alpha}:=X^{\alpha}_{0}$ and $X_{\beta}:=X^{0}_{\beta}$. If $A$ is injective then $X^{\alpha}_{\beta}$ is a Banach space with the norm
\[
\|x\|_{X^{\alpha}_{\beta}}:=\|x\|_{X}+\|\Phi^{\alpha}_{\beta}(A)^{-1}x\|_{X}=\|x\|_{X}+\|(1+A)^{\alpha+\beta}A^{-\alpha}x\|_{X}\quad(x\in X^{\alpha}_{\beta}).
\]
It follows from \cite[Proposition 3.10(i)]{BaChTo16} (the restriction $\alpha,\beta\in[0,1]$ is not needed here) that $X^{\alpha}_{\beta}=\text{ran}(A^{\alpha})\cap D(A^{\beta})$ with equivalence of norms. Finally, note that $\Phi^{\alpha}_{\beta}(A):X\to X^{\alpha}_{\beta}$ is an isomorphism. More precisely, there exists a constant $C\geq0$ such that
\begin{equation}\label{eq:norm using Phi}
\|T\|_{\La(X^{\alpha}_{\beta},X)}\leq \|T\Phi^{\alpha}_{\beta}(A)\|_{\La(X)}\leq C\|T\|_{\La(X^{\alpha}_{\beta},X)}\qquad(T\in\La(X^{\alpha}_{\beta},X)).
\end{equation}

\section{Resolvent estimates and multipliers}\label{sec:resolvents and multipliers}

In this section we prove some statements on Fourier multipliers and resolvents which will be used in later sections.

\subsection{Resolvents and Fourier multipliers}\label{subsec:resolvents and multipliers}

Throughout this subsection $-A$ is the generator of a $C_{0}$-semigroup $(T(t))_{t\geq 0}$ on a Banach space $X$.

For the reader's convenience we include a proof of the following standard lemma.

\begin{lemma}\label{lem:standardfact}
Let $n\in\N_{0}$, $x\in X$ and $\xi\in\R$. Suppose that
$-\ui\xi\in\rho(A)$ and that $[t\mapsto t^{n}T(t)x] \in L^{1}([0,\infty);X)$. Then
 \begin{align}\label{eq:equality convolution and multiplier0}
\F[t\mapsto t^{n}T(t) x](\xi) & =n!(\ui\xi+A)^{-n-1}x,\\
\label{eq:equality convolution and multiplier}
\F\Big(\int_{0}^{\infty}t^{n}T(t)g(\cdot-t)x\,\ud t\Big)(\xi) & =\widehat{g}(\xi) n!(\ui\xi+A)^{-n-1}x\qquad (g\in L^{1}(\R)).
\end{align}
\end{lemma}
\begin{proof}
It suffices to prove \eqref{eq:equality convolution and multiplier0}, as \eqref{eq:equality convolution and multiplier} follows from \eqref{eq:equality convolution and multiplier0} by standard properties of convolutions.
Since $\lambda (\lambda+A)^{-1}x\to x$ as $\lambda\to \infty$, by the dominated convergence theorem we may additionally assume that $x\in D(A)$ and that $[t\mapsto t^{n}T(t)Ax]\in L^{1}([0,\infty);X)$.
Also, \cite[Lemma 3.1.9]{vanNeerven96b} implies that $[t\mapsto T(t)x]\in C_{0}([0,\infty);X)$. Now the fundamental theorem of calculus yields
\[
(\ui\xi+A)\!\int_{0}^{\infty}\ue^{-\ui\xi t}T(t)x\,\ud t =
\Big[-\ue^{-\ui\xi t}T(t)x\Big]^{\infty}_0 = x.
\]
Hence $\int_{0}^{\infty} \ue^{-\ui\xi t}T(t)x\,\ud t=(\ui\xi+A)^{-1}x$ and
\[
\int_{0}^{\infty} \ue^{-\ui\xi t}t^{n}T(t)x\,\ud t=\frac{1}{(-\ui)^{n}}\frac{\ud^{n}}{\ud\xi^{n}}\int_{0}^{\infty} \ue^{-\ui\xi t}T(t)x\,\ud t=n!(\ui\xi+A)^{-n-1}x.\qedhere
\]
\end{proof}

We will often use the following proposition, inspired by \cite[Theorem 3.1]{Latushkin-Shvydkoy01}.

\begin{proposition}\label{prop:q-boundedness implies infty-boundedness}
Let $Y$ be a Banach space that is continuously embedded in $X$ and let $n\in\N$. Suppose that $\ui\R\setminus\{0\}\subseteq\rho(A)$ and that there exist $\psi\in L^{\infty}(\R)$, $p\in[1,\infty)$ and $q\in[1,\infty]$ such that for $j\in \{n-1, n\}\cap\N$ one has
\begin{align*}
m_{1}^{j}(\cdot) &:= \psi(\cdot)R(\ui\cdot,A)^{j}  \in\Ma_{1,\infty}(\R;\La(Y,X)),
\\ m_{2}^{j}(\cdot) & :=(1-\psi(\cdot))R(\ui\cdot,A)^{j} \in\Ma_{p,q}(\R;\La(Y,X)).
\end{align*}
Then $T_{R(\ui\cdot,A)^{n}}:L^{p}(\R;Y)\cap L^{1}(\R;Y)\to L^{\infty}(\R;X)$ is bounded and $\|T_{R(\ui\cdot,A)^{n}}\|\leq 2MC_{n}$, where $M=\sup\{\|T(t)\|_{\La(X)}\mid t\in[0,2]\}$,
\[
C_{n}=\sum_{j=n-1}^{n}\|m_{1}^j\|_{\Ma_{1,\infty}(\R;\La(Y,X))}+\|m_{2}^j\|_{\Ma_{p,q}(\R;\La(Y,X))}
\]
for $n>1$, and
\[
C_{1}= \|m_{1}^1\|_{\Ma_{1,\infty}(\R;\La(Y,X))}+\|m_{2}^1\|_{\Ma_{p,q}(\R;\La(Y,X))}+\|I_{Y}\|_{\La(Y,X)}.
\]
\end{proposition}
\begin{proof}
Let $K\in\N$, $f_{1},\ldots, f_{K}\in \dot{\Sw}(\R)$ and $x_{1},\ldots, x_{K}\in Y$, and set $f:=\sum_{k=1}^{K}f_{k}\otimes x_{k}$. Then $T_{m_{1}^{n}}(f)\in C_{b}(\R;X)$ and
\begin{equation}\label{eq:infty bound}
\sup_{t\in\R}\|T_{m_{1}^{n}}(f)(t)\|_{X}\leq \|m_{1}^{n}\|_{\Ma_{1,\infty}(\R;\La(Y,X))}\|f\|_{L^{1}(\R;Y)}.
\end{equation}
Also,
\[
\|T_{m_{2}^{n}}(f)\|_{L^{q}(\R;X)}\leq \|m_{2}^{n}\|_{\Ma_{p,q}(\R;\La(Y,X))}\|f\|_{L^{p}(\R;Y)}.
\]
The latter inequality implies that for each $l\in\Z$ there exists a $t\in[l,l+1]$ such that
\begin{equation}\label{eq:pointwise bound}
\|T_{m_{2}^{n}}(f)(t)\|_{X}\leq 2\|m_{2}^{n}\|_{\Ma_{p,q}(\R;\La(Y,X))}\|f\|_{L^{p}(\R;Y)}.
\end{equation}
Fix an $l\in\Z$ and let $t\in [l,l+1]$ be such that \eqref{eq:pointwise bound} holds. Then \eqref{eq:infty bound} implies
\begin{equation}\label{eq:pointwise bound 2}
\|T_{R(\ui\cdot,A)^{n}}(f)(t)\|_{X}\leq 2(\|m_{1}^{n}\|_{\Ma_{1,\infty}}+\|m_{2}^{n}\|_{\Ma_{p,q}})\|f\|_{L^{1}(\R;Y)\cap L^{p}(\R;Y)}.
\end{equation}
Let $\tau\in[0,2]$ and note that
\[
\ue^{\ui\xi\tau}T(\tau)R(\ui\xi,A)x=R(\ui\xi,A)x+\int_{0}^{\tau}\ue^{\ui\xi r}T(r)x\,\ud r
\]
for all $\xi\in\R\setminus\{0\}$ and $x\in X$. Hence
\begin{align*}
&T(\tau)T_{R(\ui\cdot,A)^{n}}(f)(t)=\frac{1}{2\pi}\int_{\R}\ue^{\ui\xi (t-\tau)}\ue^{\ui\xi\tau}T(\tau)R(\ui\xi,A)^{n}\widehat{f}(\xi)\,\ud\xi\\
&=\frac{1}{2\pi}\int_{\R}\ue^{\ui\xi (t-\tau)}R(\ui\xi,A)^{n}\widehat{f}(\xi)\,\ud\xi\\
& \quad +\frac{1}{2\pi}\int_{\R}\int_{0}^{\tau}\ue^{\ui\xi (t-\tau)}\ue^{\ui\xi r}T(r)R(\ui\xi,A)^{n-1}\widehat{f}(\xi)\,\ud r\ud\xi\\
&=T_{R(\ui\cdot,A)^{n}}(f)(t-\tau)+\int_{0}^{\tau}T(r)T_{R(\ui\cdot,A)^{n-1}}(f)(t-\tau+r)\,\ud r.
\end{align*}
Now \eqref{eq:pointwise bound 2} and H\"{o}lder's inequality yield
\begin{align*}
&\|T_{R(\ui\cdot,A)^{n}}(f)(t-\tau)\|_{X}\\
&\leq M\Big(\|T_{R(\ui\cdot,A)^{n}}(f)(t)\|_{X}+\int_{0}^{\tau}\|T_{R(\ui\cdot,A)^{n-1}}(f)(t-\tau+r)\|_{X}\,\ud r\Big)\\
&\leq 2M(\|m_{1}^{n}\|_{\Ma_{1,\infty}}+\|m_{2}^{n}\|_{\Ma_{p,q}})\|f\|_{L^{1}(\R;Y)\cap L^{p}(\R;Y)}\\
& \quad+M(\tau \|T_{m_{1}^{n-1}}(f)\|_{L^{\infty}(\R;X)}+\tau^{1/q'}\|T_{m_{2}^{n-1}}(f)\|_{L^{q}(\R;X)})\\
&\leq 2M\Big(\sum_{j=n-1}^{n}\|m_{1}^{j}\|_{\Ma_{1,\infty}}+\|m_{2}^{j}\|_{\Ma_{p,q}}\Big)\|f\|_{L^{p}(\R;Y)\cap L^{1}(\R;Y)}
\end{align*}
for $n>1$. For $n=1$ the computation is similar, but one can directly estimate
\[
\int_{0}^{\tau}\|f(t-\tau-r)\|_{X}\,\ud r\leq \|I_{Y}\|_{\La(Y,X)}\|f\|_{L^{1}(\R;Y)}.
\]
This concludes the proof, since $\tau\in[0,2]$ and $l\in\Z$ are arbitrary and since $\dot{\Sw}(\R)\otimes Y\subseteq L^{p}(\R;Y)\cap L^{1}(\R;Y)$ is dense.
\end{proof}

\begin{remark}\label{rem:other multiplier conditions at zero}
When applying Proposition \ref{prop:q-boundedness implies infty-boundedness} we will consider $\psi$ with compact support. Then one may assume that $m_{1}^{j}\in \Ma_{u,v}(\R;\La(Y,X))$ for general $u\in[1,\infty)$ and $v\in[1,\infty]$. For $\chi\in C_{c}^{\infty}(\R)$ such that $\chi\equiv 1$ on $\supp(\psi)$ one has $m_{1}^{j}=\chi m_{1}^{j}\in \Ma_{u,\infty}(\R;\La(Y,X))$ by Young's inequality. The same proof now shows that $T_{R(\ui\cdot,A)^{n}}:L^{u}(\R;Y)\cap L^{p}(\R;Y)\to L^{\infty}(\R;X)$ is bounded, with
\[
\|T_{R(\ui\cdot,A)^{n}}\|\leq 2M\Big(\!\sum_{j=n-1}^{n}\|m_{1}^j\|_{\Ma_{u,v}(\R;\La(Y,X))}+\|m_{2}^j\|_{\Ma_{p,q}(\R;\La(Y,X))}\Big)
\]
for $n>1$, and similarly for $n=1$. However, Young's inequality also shows that $m_{1}^{j}=\chi m_{1}^{j}\chi\in\Ma_{1,\infty}(\R;\La(Y,X))$, so that these assumptions are no more general than those in Proposition \ref{prop:q-boundedness implies infty-boundedness}.
\end{remark}

\subsection{Resolvent estimates}\label{subsec:resolvent estimates}

We now present two propositions on resolvent growth. The assertions on uniform boundedness have for the most part been obtained by different methods in \cite[Lemma 3.3]{Weis97}, \cite[Lemma 1.1]{Huang-vanNeerven99}, \cite[Lemma 3.2]{Latushkin-Shvydkoy01} and \cite[Theorem 5.5]{BaChTo16}. The proof below allows us to also deduce the corresponding statements on $R$-boundedness directly. Note that if $A$ satisfies \eqref{eq:equivalence growth bounds 1a} with $\alpha\in(0,1)$ then one may in fact let $\alpha=0$, by elementary properties of resolvents.

\begin{proposition}\label{prop:app-equivalence growth bounds}
Let $\alpha\in\{0\}\cup[1,\infty)$, $\beta\in[0,\infty)$ and $\beta_{0}\in[0,1]$, and let $A$ be an injective sectorial operator on a Banach space $X$. Let $\ph\in(0,\tfrac{\pi}{2}]$ and $\Omega:=\overline{\C_{+}}\setminus (S_{\ph}\cup\{0\})$, and suppose that $-\Omega\subseteq\rho(A)$. Then the following statements hold:
\begin{enumerate}
\item\label{item:app-equivalence growth bounds 1} The collection
\begin{equation}\label{eq:equivalence growth bounds 1a}
\{\lambda^{\alpha}(\lambda+A)^{-1}\mid \lambda\in \Omega, \abs{\lambda}\leq 1\}\subseteq\La(X)
\end{equation}
is uniformly bounded if and only if
\begin{equation}\label{eq:equivalence growth bounds 1b}
\{(\lambda+A)^{-1}\mid \lambda\in \Omega,\abs{\lambda}\leq 1\}\subseteq \La(X^{\alpha},X)
\end{equation}
is uniformly bounded. Moreover, \eqref{eq:equivalence growth bounds 1a} is $R$-bounded if and only if \eqref{eq:equivalence growth bounds 1b} is $R$-bounded.
\item\label{item:app-equivalence growth bounds 2} The collection
\begin{equation}\label{eq:equivalence growth bounds 2}
\{\lambda^{-\beta}(\lambda+A)^{-1}\mid \lambda\in \Omega, \abs{\lambda}\geq 1\}\subseteq\La(X)
\end{equation}
is uniformly bounded if and only if
\begin{equation}\label{eq:equivalence growth bounds 3}
\{ \lambda^{\beta_0} (\lambda+A)^{-1}\mid \lambda\in \Omega,\abs{\lambda}\geq 1\}\subseteq \La(X_{\beta+\beta_0},X)
\end{equation}
is uniformly bounded. Moreover, \eqref{eq:equivalence growth bounds 2} is $R$-bounded if and only if \eqref{eq:equivalence growth bounds 3} is $R$-bounded.
\item\label{item:app-equivalence growth bounds 3} The collection
\[
\Big\{(1-\lambda)^{\beta_{0}}(\lambda+A)^{-1} A^{\alpha} (1+A)^{-\alpha-\beta-\beta_{0}}-\frac{(-\lambda)^{\alpha}}{(1-\lambda)^{\alpha+\beta}} (\lambda+A)^{-1}\Big| \lambda\in\Omega\Big\}
\]
is $R$-bounded in $\La(X)$.
\end{enumerate}
\end{proposition}
\begin{proof}
Fix $\theta\in(\max(\w_{A},\pi-\ph),\pi)$ and let $\Gamma:=\{r\ue^{\ui\theta}\mid r\in[0,\infty)\}\cup\{r\ue^{-\ui\theta}\mid r\in[0,\infty)\}$ be oriented from $\infty\ue^{\ui\theta}$ to $\infty \ue^{-\ui\theta}$.

For \eqref{item:app-equivalence growth bounds 1} first note that, by the resolvent identity,
\begin{align*}
(\lambda+A)^{-1}A(1+A)^{-1}&=(1+A)^{-1}-\lambda(\lambda+A)^{-1}(1+A)^{-1}\\
&=(1+A)^{-1}-\frac{\lambda}{1+\lambda}(\lambda+A)^{-1}-\frac{\lambda}{1+\lambda}(1+A)^{-1}\\
&=\frac{1}{1+\lambda}(1+A)^{-1}-\frac{\lambda}{1+\lambda}(\lambda+A)^{-1}
\end{align*}
for all $\lambda\in \Omega$. Now \eqref{eq:Kahane contraction} and \eqref{eq:norm using Phi} yield \eqref{item:app-equivalence growth bounds 1} for $\alpha=1$.

Let $\alpha>1$. Then
\begin{equation}\label{eq:equivalence growth bounds hulp5}
\begin{aligned}
&(\lambda+A)^{-1}A^{\alpha}(1+A)^{-\alpha}=(\lambda+A)^{-1}(1+A)A^{\alpha}(1+A)^{-\alpha-1}\\
&=A^{\alpha}(1+A)^{-\alpha-1}+(1-\lambda)(\lambda+A)^{-1}A^{\alpha}(1+A)^{-\alpha-1}
\end{aligned}
\end{equation}
for all $\lambda\in\Omega$. Since the singleton $\{A^{\alpha}(1+A)^{-\alpha-1}\}\subseteq\La(X)$ is $R$-bounded, by \eqref{eq:norm using Phi} it suffices to show that \eqref{eq:equivalence growth bounds 1a} is uniformly bounded (or $R$-bounded) if and only if
\begin{equation}\label{eq:equivalence growth bounds 4}
\{(1-\lambda)(\lambda+A)^{-1}A^{\alpha}(1+A)^{-\alpha-1}\mid \lambda\in \Omega, \abs{\lambda}\leq 1\}\subseteq\La(X)
\end{equation}
is uniformly bounded (or $R$-bounded). The resolvent identity and \eqref{eq:fractional power} yield
\begin{align*}
(\lambda+A)^{-1}A^{\alpha}(1+A)^{-\alpha-1}&=\frac{1}{2\pi\ui}\int_{\Gamma}\frac{z^{\alpha}}{(1+z)^{\alpha+1}}(\lambda+A)^{-1}R(z,A)\,\ud z\\
&=\frac{1}{2\pi\ui}\int_{\Gamma}\frac{z^{\alpha}}{(1+z)^{\alpha+1}(z+\lambda)}\,\ud z (\lambda+A)^{-1}\\
&\quad +\frac{1}{2\pi\ui}\int_{\Gamma}\frac{z^{\alpha}}{(1+z)^{\alpha+1}(z+\lambda)}R(z,A)\,\ud z
\end{align*}
for $\lambda\in\Omega$. Hence, using \eqref{eq:app-auxiliary equation 1} of Lemma \ref{lem:app-auxiliary lemma},
\begin{equation}\label{eq:equivalence growth bounds 5}
(1-\lambda)(\lambda+A)^{-1}A^{\alpha}(1+A)^{-\alpha-1}=\frac{(-\lambda)^{\alpha}}{(1-\lambda)^{\alpha}}(\lambda+A)^{-1}+S_{\lambda},
\end{equation}
where
\[
S_{\lambda}:=\frac{1}{2\pi\ui}\int_{\Gamma}\frac{z^{\alpha}}{(1+z)^{\alpha+1}}\frac{1-\lambda}{z+\lambda}R(z,A)\,\ud z.
\]
Now fix $\varepsilon \in(0, \min(\alpha-1,1)]$. Then $z\mapsto \frac{z^{\varepsilon}}{(1+z)^{2\varepsilon}}R(z,A)$ is integrable on $\Gamma$, and
\[
\sup\Big\{\frac{\abs{z}^{\alpha-\varepsilon}}{\abs{1+z}^{\alpha+1-2\varepsilon}}\frac{\abs{1-\lambda}}{\abs{z+\lambda}}\Big|\lambda\in\Omega,z\in\Gamma\Big\}<\infty
\]
by \eqref{eq:app-auxiliary equation 2} in Lemma \ref{lem:app-auxiliary lemma}. Hence
\cite[Corollary 2.17]{Kunstmann-Weis04} implies that $\{S_{\lambda}\mid \lambda\in\Omega\}\subseteq\La(X)$ is $R$-bounded. Now \eqref{eq:equivalence growth bounds 5} shows that the uniform boundedness (or $R$-boundedness) of \eqref{eq:equivalence growth bounds 1a} and \eqref{eq:equivalence growth bounds 4} are equivalent, thereby proving \eqref{item:app-equivalence growth bounds 1}.

For \eqref{item:app-equivalence growth bounds 2} we may suppose that $\beta+\beta_{0}>0$. Then \eqref{eq:fractional power}, applied to the invertible sectorial operator $\tfrac{1}{2}+A$, and the resolvent identity imply that
\begin{align*}
(\lambda+A)^{-1}(1+A)^{-\beta-\beta_{0}}&=\frac{1}{2\pi\ui}\int_{\Gamma}\frac{1}{(\frac{1}{2}+z)^{\beta+\beta_0}}(\lambda+A)^{-1}R(z,\tfrac{1}{2}+A)\,\ud z\\
&=\frac{1}{2\pi\ui}\int_{\Gamma}\frac{1}{(\frac{1}{2}+z)^{\beta+\beta_0}(z+\lambda-\frac{1}{2})}\,\ud z(\lambda+A)^{-1}\\
&\quad +\frac{1}{2\pi\ui}\int_{\Gamma}\frac{1}{(\frac{1}{2}+z)^{\beta+\beta_0}(z+\lambda-\frac{1}{2})}R(z,\tfrac{1}{2}+A)\,\ud z
\end{align*}
for $\lambda\in\Omega$. Now \eqref{eq:app-auxiliary equation 1} yields
\begin{equation}\label{eq:equivalence growth bounds 6}
(1-\lambda)^{\beta_{0}}(\lambda+A)^{-1}(1+A)^{-\beta-\beta_{0}}=\frac{1}{(1-\lambda)^{\beta}}(\lambda+A)^{-1}+(1-\lambda)^{\beta_{0}}T_{\lambda},
\end{equation}
where
\[
T_{\lambda}:=\frac{1}{2\pi\ui}\int_{\Gamma}\frac{1}{(\frac{1}{2}+z)^{\beta+\beta_{0}}(z+\lambda-\frac{1}{2})}R(z,\tfrac{1}{2}+A)\,\ud z.
\]
Fix $\varepsilon\in(0,\beta+\beta_{0})$. Then $z\mapsto (z+\frac12)^{-\varepsilon} R(z,\tfrac{1}{2}+A)$ is integrable on $\Gamma$, and
\[
\sup\Big\{\frac{1+\abs{\lambda}}{\abs{\frac{1}{2}+z}^{\beta+\beta_{0}-\varepsilon}\abs{z+\lambda-\frac{1}{2}}}\Big|\lambda\in\Omega,z\in\Gamma\Big\}<\infty
\]
by \eqref{eq:app-auxiliary equation 2}. Hence \cite[Corollary 2.17]{Kunstmann-Weis04} implies that $\{(1+\abs{\lambda})T_{\lambda}\mid \lambda\in\Omega\}$ is $R$-bounded.
Since $|1-\lambda|^{\beta_0}\leq 1+\abs{\lambda}$ for all $\lambda\in\Omega$, the proof of part \eqref{item:app-equivalence growth bounds 2} is completed using \eqref{eq:Kahane contraction}, \eqref{eq:equivalence growth bounds 6} and \eqref{eq:norm using Phi}.

Finally, for \eqref{item:app-equivalence growth bounds 3} we restrict to the 
case where $\alpha>1$ and $\beta>0$. The other cases follow in a similar manner from the proofs of \eqref{item:app-equivalence growth bounds 1} and \eqref{item:app-equivalence growth bounds 2}.
%several different cases: $(\alpha, \beta)$ is one of the form $(1,\beta)$ and %$(\alpha, \beta)$ with $\alpha>1$ and $\beta>0$. The remaining cases $(\alpha, 0)$ and $(0,\beta)$ are simpler. We give the details of the proof in the most involved %case $(\alpha, \beta)$ with $\alpha>1$ and $\beta>0$.
%In this case we have to combine the identities in \eqref{item:app-equivalence growth bounds 1} and \eqref{item:app-equivalence growth bounds 2}. 
The operator family in \eqref{item:app-equivalence growth bounds 3} can be written as
\begin{align*}
&A^{\alpha}(1+A)^{-\alpha} \Big[ (1-\lambda)^{\beta_0} (\lambda+A)^{-1} (1+A)^{-\beta-\beta_{0}} -  (1-\lambda)^{-\beta}(\lambda+A)^{-1}\Big]\\
& + (1-\lambda)^{-\beta}\Big[(\lambda+A)^{-1} A^{\alpha}(1+A)^{-\alpha} - \frac{(-\lambda)^{\alpha}}{(1-\lambda)^{\alpha}} (\lambda+A)^{-1}  \Big]
\\
&=: A^{\alpha}(1+A)^{-\alpha} V_{\lambda}^{1} + (1-\lambda)^{-\beta} V_{\lambda}^2.
\end{align*}
Using standard algebraic properties of $R$-boundedness (see \cite[Proposition 8.1.19]{HyNeVeWe2}), it suffices to prove that $\{V_{\lambda}^i\mid \lambda\in \Omega\}\subseteq\La(X)$ is $R$-bounded for $i\in\{1,2\}$. The proof of \eqref{item:app-equivalence growth bounds 2}, and in particular \eqref{eq:equivalence growth bounds 6}, shows that
\[
R(\{V_{\lambda}^1\mid \lambda\in \Omega\}) = R(\{(1-\lambda)^{\beta_{0}}T_{\lambda} \mid \lambda\in \Omega \})<\infty.
\]
For the other term note that, by \eqref{eq:equivalence growth bounds hulp5} and \eqref{eq:equivalence growth bounds 5}, $V_{\lambda}^2 = A^{\alpha}(1+A)^{-\alpha-1}+S_{\lambda}$. 
Hence the proof of \eqref{item:app-equivalence growth bounds 1} yields 
\[
R(\{V_{\lambda}^2\mid\lambda\in \Omega\})\leq \|A^{\alpha}(1+A)^{-\alpha-1}\|_{\La(X)} + R(\{S_{\lambda}\mid\lambda\in \Omega\})<\infty.\qedhere
\]
\end{proof}

\begin{corollary}\label{cor:app-growth and fractional domains}
Let $\alpha\in [0,\infty)$ and $\alpha_{0}\in[0,\alpha]$. Let $A$ be an injective sectorial operator on a Banach space $X$ such that $\ui\R\setminus\{0\}\subseteq\rho(A)$ and
\[
\sup\{\|\lambda^{\alpha}(\lambda+A)^{-1}\|_{\La(X)}\mid \lambda\in \ui\R\setminus\{0\}, \abs{\lambda}\leq 1\}<\infty.
\]
Then
\[
\sup\{\|\lambda^{\alpha-\alpha_{0}}(\lambda+A)^{-1}\|_{\La(X^{\alpha_{0}},X)}\mid \lambda\in\ui\R\setminus\{0\}, \abs{\lambda}\leq 1\}<\infty.
\]
\end{corollary}
\begin{proof}
First note that $0\in\rho(A)$ for $\alpha<1$, by elementary properties of resolvents. Hence, by Proposition \ref{prop:app-equivalence growth bounds} \eqref{item:app-equivalence growth bounds 1} it suffices to consider $\alpha\geq 1$ and $\alpha_{0}\in(0,\alpha)$. By \cite[Propositions 2.1.1.f and 3.1.9]{Haase06a}, $A(1+A)^{-1}$ is a sectorial operator and $A^{\alpha_{0}}(1+A)^{-\alpha_{0}}=(A(1+A)^{-1})^{\alpha_{0}}$. Now the moment inequality \cite[Proposition 6.6.4]{Haase06a} and another application of \cite[Proposition 3.1.9]{Haase06a} yield
\begin{align*}
&\|\lambda^{\alpha-\alpha_{0}}(\lambda+A)^{-1}A^{\alpha_{0}}(1+A)^{-\alpha_{0}}x\|_{X}\\
&=\abs{\lambda}^{\alpha-\alpha_{0}}\|(A(1+A)^{-1})^{\alpha_{0}}(\lambda+A)^{-1}x\|_{X}\\
&\lesssim\abs{\lambda}^{\alpha-\alpha_{0}}\|(\lambda+A)^{-1}(A(1+A)^{-1})^{\alpha}x\|_{X}^{\alpha_{0}/\alpha}\|(\lambda+A)^{-1}x\|_{X}^{(\alpha-\alpha_{0})/\alpha}\\
&\leq \|(\lambda+A)^{-1}A^{\alpha}(1+A)^{-\alpha}\|_{\La(X)}^{\alpha_{0}/\alpha}\|\lambda^{\alpha}(\lambda+A)^{-1}\|_{\La(X)}^{(\alpha-\alpha_{0})/\alpha}\|x\|_{X}
\end{align*}
for all $\lambda\in\ui\R\setminus\{0\}$ and $x\in X$. Proposition \ref{prop:app-equivalence growth bounds} \eqref{item:app-equivalence growth bounds 1} and \eqref{eq:norm using Phi} conclude the proof.
\end{proof}

\section{Polynomial stability}\label{sec:polynomial stability}

In this section we study polynomial stability for semigroups using Fourier multipliers. We first obtain some results valid on general Banach spaces. Then we establish the connection between polynomial stability and Fourier multipliers, and we use this link to deduce polynomial stability results under geometric assumptions on the underlying space. We also study the necessity of the spectral assumptions which we make, compare our theorems with the literature, and give examples to illustrate our results.

The following terminology will be used throughout this section.

\begin{definition}\label{def:operators with resolvent growth}
Let $\alpha,\beta\in[0,\infty)$. An operator $A$ on a Banach space $X$ has \emph{resolvent growth $(\alpha,\beta)$} if the following conditions hold:
\begin{enumerate}[(i)]
\item\label{item:semigroup assumption} $-A$ generates a $C_{0}$-semigroup $(T(t))_{t\geq 0}$ on $X$;
\item\label{item:growth assumption} $\overline{\C_{-}}\setminus\{0\}\subseteq\rho(A)$, and
\[
\Big\{\frac{\lambda^{\alpha}}{(1+\lambda)^{\alpha+\beta}}(\lambda+A)^{-1}\Big|\, \lambda\in \overline{\C_{+}}\setminus\{0\}\Big\}\subseteq\La(X)
\]
is uniformly bounded.
\end{enumerate}
An operator $A$ has \emph{$R$-resolvent growth $(\alpha,\beta)$} if $A$ has resolvent growth $(\alpha,\beta)$ and
\[
\Big\{\lambda^{-\beta}(\lambda+A)^{-1}\Big|\, \lambda\in \overline{\C_{+}}, \abs{\lambda}\geq 1\Big\}\subseteq\La(X)
\]
is $R$-bounded.
\end{definition}

Note that we do not assume in \eqref{item:semigroup assumption} that the semigroup generated by $-A$ is uniformly bounded. We will implicitly use throughout that each operator $A$ with resolvent growth $(\alpha,\beta)$, for $\alpha\in[0,1)$ and $\beta\in[0,\infty)$, is invertible and thus has resolvent growth $(0,\beta)$, as follows from the fact that $\|R(\lambda,A)\|_{\La(X)}\geq \text{dist}(\lambda,\sigma(A))^{-1}$ for all $\lambda\in\rho(A)$.

Recall that we use the convention that $\frac{0}{0}=\infty$, for simplicity of notation.

\subsection{Results on general Banach spaces}

The following lemma is used to interpolate between decay rates. Related results can be found in \cite[Proposition 3.1]{BaEnPrSc06} and \cite[Lemma 4.2]{BaChTo16}. Recall the definition of the space $X^{\alpha}_{\beta}$, for $\alpha,\beta\geq0$, from Section \ref{subsec:sectorial operators}.

\begin{lemma}\label{lem:interpolation}
Let $A$ be an injective sectorial operator on a Banach space $X$ such that $-A$ generates a $C_{0}$-semigroup $(T(t))_{t\geq 0}$ on $X$. For $j\in\{1,2\}$ let $\alpha_{j},\beta_{j}\in[0,\infty)$ be such that $\alpha_{1}\geq \alpha_{2}$ and $\beta_{1}\geq \beta_{2}$, and let $f_{j}:[0,\infty)\to [0,\infty)$ be such that $\|T(t)\|_{\La(X^{\alpha_{j}}_{\beta_{j}},X)}\leq f_{j}(t)$ for all $t\in[0,\infty)$. Then for each $\theta\in[0,1]$ there exists a $C_{\theta}\in[0,\infty)$ such that
\begin{equation}\label{eq:decay on intermediate domains}
\|T(t)\|_{\calL(X^{\theta\alpha_{1} + (1-\theta)\alpha_2}_{\theta\beta_{1} + (1-\theta)\beta_2},X)}\leq C_{\theta} (f_{1}(t))^{\theta} (f_{2}(t))^{1-\theta}\qquad (t\in[0,\infty)).
\end{equation}
Moreover, suppose that $f_{1}(t)=Ct^{-\mu}$ for some $C,\mu\in[0,\infty)$ and all $t\in[1,\infty)$. Then for each $\theta\in[1,\infty)$ there exists a $C_{\theta}\in[0,\infty)$ such that
\begin{equation}\label{eq:decay on larger domains}
\|T(t)\|_{\La(X^{\theta\alpha_{1}}_{\theta\beta_{1}},X)}\leq C_{\theta} t^{-\mu\theta}\qquad(t\in[1,\infty)).
\end{equation}
\end{lemma}
\begin{proof}
Let $t\in[0,\infty)$ and note that, by \eqref{eq:norm using Phi} and \eqref{eq:functional calculus},
\begin{align*}
\|T(t)\|_{\La(X^{\theta\alpha_{1}+(1-\theta)\alpha_{2}}_{\theta\beta_{1}+(1-\theta)\beta_{2}},X)}&\leq \|T(t)\Phi^{\theta\alpha_{1}+(1-\theta)\alpha_{2}}_{\theta\beta_{1}+(1-\theta)\beta_{2}}(A)\|_{\La(X)}\\
&=\|T(t)\Phi^{\theta(\alpha_{1}-\alpha_{2})}_{\theta(\beta_{1}-\beta_{2})}(A)\Phi^{\alpha_{2}}_{\beta_{2}}(A)\|_{\La(X)}.
\end{align*}
Let $c:=\alpha_{1}-\alpha_{2}+\beta_{1}-\beta_{2}$. Then $\Phi^{(\alpha_{1}-\alpha_{2})/c}_{(\beta_{1}-\beta_{2})/c}(A)=A^{(\alpha_{1}-\alpha_{2})/c}(1+A)^{-1}$ is sectorial, by \cite[Proposition 3.10]{BaChTo16}. Hence \cite[Theorem 2.4.2]{Haase06a} yields
\begin{align*}
\Phi^{\theta(\alpha_{1}-\alpha_{2})}_{\theta(\beta_{1}-\beta_{2})}(A)=A^{\theta(\alpha_{1}-\alpha_{2})}(1+A)^{-\theta(\alpha_{1}-\alpha_{2}+\beta_{1}-\beta_{2})}=(\Phi^{(\alpha_{1}-\alpha_{2})/c}_{(\beta_{1}-\beta_{2})/c}(A))^{c\theta}.
\end{align*}
The moment inequality \cite[Proposition 6.6.4]{Haase06a} and \cite[Theorem 2.4.2]{Haase06a} imply that
\begin{align*}
\|(\Phi^{(\alpha_{1}-\alpha_{2})/c}_{(\beta_{1}-\beta_{2})/c}(A))^{c\theta}x\|_{X}\lesssim \|(\Phi^{(\alpha_{1}-\alpha_{2})/c}_{(\beta_{1}-\beta_{2})/c}(A))^{c}x\|_{X}^{\theta}\|x\|_{X}^{1-\theta}=\|\Phi^{\alpha_{1}-\alpha_{2}}_{\beta_{1}-\beta_{2}}(A)x\|_{X}^{\theta}\|x\|_{X}^{1-\theta}
\end{align*}
for $x\in D(\Phi^{\alpha_{1}-\alpha_{2}}_{\beta_{1}-\beta_{2}}(A))$. Combining all this with \eqref{eq:functional calculus} and \eqref{eq:norm using Phi} shows that
\begin{align*}
&\|T(t)\|_{\La(X^{\theta\alpha_{1}+(1-\theta)\alpha_{2}}_{\theta\beta_{1}+(1-\theta)\beta_{2}},X)}\leq \|T(t)\Phi^{\theta(\alpha_{1}-\alpha_{2})}_{\theta(\beta_{1}-\beta_{2})}(A)\Phi^{\alpha_{2}}_{\beta_{2}}(A)\|_{\La(X)}\\
&=\|(\Phi^{(\alpha_{1}-\alpha_{2})/c}_{(\beta_{1}-\beta_{2})/c}(A))^{c\theta}T(t)\Phi^{\alpha_{2}}_{\beta_{2}}(A)\|_{\La(X)}\\
&\lesssim \|\Phi^{\alpha_{1}-\alpha_{2}}_{\beta_{1}-\beta_{2}}(A)T(t)\Phi^{\alpha_{2}}_{\beta_{2}}(A)\|_{\La(X)}^{\theta}\|T(t)\Phi^{\alpha_{2}}_{\beta_{2}}(A)\|_{\La(X)}^{1-\theta}\\
&=\|T(t)\Phi^{\alpha_{1}}_{\beta_{1}}(A)\|_{\La(X)}^{\theta}\|T(t)\Phi^{\alpha_{2}}_{\beta_{2}}(A)\|_{\La(X)}^{1-\theta}\\
&\lesssim\|T(t)\|_{\La(X^{\alpha_{1}}_{\beta_{1}},X)}^{\theta}\|T(t)\|_{\La(X^{\alpha_{2}}_{\beta_{2}},X)}^{1-\theta}\leq (f_{1}(t))^{\theta}(f_{2}(t))^{1-\theta},
\end{align*}
thereby proving \eqref{eq:decay on intermediate domains}. As for \eqref{eq:decay on larger domains}, let $n\in\N$. Then
\begin{align*}
\|T(t)\|_{\La(X^{n\alpha_{1}}_{n\beta_{1}},X)} & \leq \|T(t)\Phi^{n\alpha_{1}}_{n\beta_{1}}(A)\|_{\La(X)}\leq \|T(\tfrac{t}{n})\Phi^{\alpha_{1}}_{\beta_{1}}(A)\|_{\La(X)}^n
\\ & \lesssim (f_1(\tfrac{t}{n}))^n = C^n n^{\mu n} t^{-\mu n},
\end{align*}
which implies \eqref{eq:decay on larger domains} for $\theta\in\N$. Finally, applying \eqref{eq:decay on intermediate domains} to interpolate between $(n\alpha_1,n\beta_1)$ and $((n+1)\alpha_1,(n+1)\beta_1)$ yields \eqref{eq:decay on larger domains} for all $\theta\in[1,\infty)$.
\end{proof}

The following result for $C_0$-semigroups on general Banach spaces extends \cite[Proposition 3.4]{BaEnPrSc06}, where the case $\alpha=\rho=0$ was considered.

\begin{proposition}\label{prop:general semigroups}
Let $\alpha,\beta\in[0,\infty)$ and let $A$ be an injective sectorial operator with resolvent growth $(\alpha,\beta)$ on a Banach space $X$. Let $\sigma,\tau\in[0,\infty)$ be such that $\sigma>\alpha-1$ and $\tau>\beta+1$. Then for each $\rho\in[0,\min(\frac{\sigma+1}{\alpha}-1,\frac{\tau-1}{\beta}-1))$ there exists a $C_{\rho}\in[0,\infty)$ such that
\begin{equation}\label{eq:general semigroups main}
\|T(t)\|_{\La(X^{\sigma}_{\tau},X)}\leq C_{\rho}t^{-\rho}\qquad (t\in[1,\infty)).
\end{equation}
\end{proposition}
\begin{proof}
By elementary calculations the proposition is equivalent to the following statement: for all $s\geq0$ and $\delta,\veps>0$ there exists a $C_{s,\delta,\veps}\geq0$ such that
\begin{equation}\label{eq:general semigroups}
\|T(t)\|_{\La(X^{\mu}_{\nu},X)}\leq C_{s,\delta,\veps}t^{-s}  \qquad (t\in [1, \infty)),
\end{equation}
where $\mu= \max((s+1) \alpha -1 + \delta,0)$  and $\nu = (s+1)\beta+1+\veps$. Furthermore, by Lemma \ref{lem:interpolation} it suffices to prove \eqref{eq:general semigroups} for $n:=s\in \N_0$.

Let $x\in X^{\mu}_{\nu+1}$ and set $y:=\Phi^{\mu}_{\nu}(A)x=A^{-\mu} (1+A)^{\mu+\nu}x\in D(A)$. Then
\[
g(t):=\frac{1}{2\pi\ui}\int_{\ui\infty}^{-\ui\infty}\!\ue^{-\lambda t}\frac{\lambda^{\mu}}{(1+\lambda)^{\mu+\nu}}R(\lambda,A)y\,\ud\lambda
\]
is a well defined element of $X$ for all $t\geq0$. One can check that $g$ is continuously differentiable with $g'(t)=-Ag(t)$. Also,
\begin{align*}
g(0)&=\frac{1}{2\pi\ui}\int_{\ui\infty}^{-\ui\infty}\frac{\lambda^{\mu}}{(1+\lambda)^{\mu+\nu}}R(\lambda,A)y\,\ud\lambda=A^{\mu}(1+A)^{-\mu-\nu}y=x.
\end{align*}
Here we have deformed the path of integration to the curve $\Gamma=\{r\ue^{\ui\theta}\mid r\in[0,\infty)\}\cup\{r\ue^{-\ui\theta}\mid r\in[0,\infty)\}$ in \eqref{eq:fractional power}, for $\theta\in(\w_{A},\pi)$, which we may do by the assumptions on $A$. Now $g(t)=T(t)x$, by uniqueness of the Cauchy problem associated with $-A$. Integration by parts yields
\begin{align*}
t^{n}T(t)x&=\frac{t^{n}}{2\pi\ui}\int_{\ui\R}\ue^{-\lambda t}\frac{\lambda^{\mu}}{(1+\lambda)^{\mu+\nu}}R(\lambda,A)y\,\ud\lambda\\
&=\frac{(-1)^{n}}{2\pi\ui}\int_{\ui\R}\Big(\frac{d^{n}}{d\lambda^{n}}\ue^{-\lambda t}\Big)\frac{\lambda^{\mu}}{(1+\lambda)^{\mu+\nu}}R(\lambda,A)y\,\ud\lambda\\
&=\frac{1}{2\pi\ui}\int_{\ui\R}\ue^{-\lambda t}p(\lambda,A)y\,\ud\lambda.
\end{align*}
Here $p(\lambda,A)$ is a finite linear combination of terms of the form
\[
\frac{\lambda^{\mu-j}}{(1+\lambda)^{\mu+\nu + (k-j)}}R(\lambda,A)^{n-k+1}
\]
for $0\leq j\leq k\leq n$, where we let $j=0$ if $\mu = 0$. Then
\begin{align*}
\|t^{n}T(t)x\|_{X}&\leq \frac{1}{2\pi}\int_{\ui\R}\|p(\lambda,A)\|_{\La(X)}\|y\|_{X}\ud\abs{\lambda}\lesssim \|(-A)^{-\mu}(1+A)^{\mu+\nu}x\|_{X}\leq\|x\|_{X^{\mu}_{\nu}}
\end{align*}
with implicit constants independent of $t$ and $x$. Since $X^{\mu}_{\nu+1}$ is dense in $X^{\mu}_{\nu}$, the proof is concluded.
\end{proof}

The following corollary of Proposition \ref{prop:general semigroups} and Lemma \ref{lem:interpolation} takes into account the growth behavior of $(T(t))_{t\geq 0}$ on $X$. It also extends Proposition \ref{prop:general semigroups} by providing stability rates on $X^{\sigma}_{\tau}$ for $\sigma\in [0,\alpha-1]$ and $\tau\in [0,\beta+1]$. The same approach was used in \cite[Theorem 3.5]{BaEnPrSc06} for uniformly bounded semigroups and $\alpha=0$.

\begin{corollary}\label{cor:growth of semigroup on X}
Let $\alpha,\beta\in[0,\infty)$ and let $A$ be an injective sectorial operator with resolvent growth $(\alpha,\beta)$ on a Banach space $X$. Let $\sigma,\tau\in[0,\infty)$.
Then for each $\rho\in[0,\min(\frac{\sigma}{\alpha},\frac{\tau}{\beta}))$ there exists a $C_{\rho}\in[0,\infty)$ such that
\begin{equation}\label{eq:growth of semigroup on X main}
\|T(t)\|_{\La(X^{\sigma}_{\tau},X)} \leq C_{\rho} \max(1,\|T(t)\|_{\La(X)})t^{-\rho} \qquad (t\in[1,\infty)).
\end{equation}
\end{corollary}
\begin{proof}
By elementary calculations it suffices to prove the following: for all $s\geq0$ and $\delta,\varepsilon>0$ there exists a constant $C_{s,\delta,\veps}\geq0$ such that
\begin{equation}\label{eq:growth of semigroup on X}
\|T(t)\|_{\La(X^{\mu}_{\nu},X)} \leq C_{s,\delta,\veps} \max(1,\|T(t)\|_{\La(X)})t^{-s} \qquad (t\in[1,\infty)),
\end{equation}
where $\mu = s\alpha+\delta$ and $\nu = s\beta+\veps$. Let $\wt{\veps}>0$ and for $\theta\in (0,1)$ set $\wt{s} := s/\theta$, $\wt{\mu}:= \max((\wt{s}+1) \alpha -1 + \wt{\veps},0)$  and $\wt{\nu} := (\wt{s}+1)\beta+1+\wt{\veps}$.
Then, by Lemma \ref{lem:interpolation} and \eqref{eq:general semigroups},
\[
\|T(t)\|_{\La(X^{\wt{\mu}\theta}_{\wt{\nu}\theta},X)}\lesssim \|T(t)\|_{\calL(X)}^{1-\theta} \|T(t)\|_{\calL(X^{\wt{\mu}}_{\wt{\nu}},X)}^{\theta} \lesssim \max(1,\|T(t)\|_{\La(X)}) t^{-s}
\]
for all $t\geq1$. Next, note that $\wt{\mu}\theta = \max(s\alpha+\theta(\alpha-1+\wt{\varepsilon}), 0)$ and $\wt{\nu}\theta = s\beta + \theta(\beta+1+\wt{\varepsilon})$.
Now the proof is concluded by letting $\theta\in (0,1)$ be such that $\wt{\mu}\theta\leq s\alpha+\varepsilon$ and $\wt{\nu}\theta\leq s\beta+\varepsilon$.
\end{proof}

\subsection{Polynomial stability and Fourier multipliers}\label{subsec:polynomial stability and multipliers}

In this subsection we relate polynomial stability of a semigroup to Fourier multiplier properties of the resolvent of its generator.

In order to state our abstract result on polynomial stability we introduce a class of admissible spaces.

\begin{definition}\label{def:admissable spaces}
Let $-A$ be the generator of a $C_{0}$-semigroup $(T(t))_{t\geq 0}$ on a Banach space $X$, and let $n\in\N_{0}$. A Banach space $Y$ which is continuously embedded in $X$ is \emph{$(A,n)$-admissible} if the following conditions hold:
\begin{enumerate}[(i)]
\item there exists a constant $C_{T}\in[0,\infty)$ such that $T(t)Y\subseteq Y$ and
\[
\|T(t)\restriction_{Y}\|_{\La(Y)}\leq C_{T}\|T(t)\|_{\La(X)}\qquad (t\in[0,\infty));
\]
\item\label{item:density admissable space} there exists a dense subspace $Y_{0}\subseteq Y$ such that $[t\mapsto t^{n}T(t)y]\in L^{1}([0,\infty);X)$ for all $y\in Y_{0}$.
\end{enumerate}
\end{definition}

Let $\alpha,\beta\in[0,\infty)$ and let $A$ be an injective sectorial operator with resolvent growth $(\alpha,\beta)$. Then $Y=X^{\sigma}_{\tau}$ is $(A,n)$-admissible for all $\sigma,\tau\in[0,\infty)$ and $n\in\N_{0}$, by Proposition \ref{prop:general semigroups}.

The following theorem is our main result relating polynomial stability and Fourier multipliers. It follows from \eqref{eq:quantitative infty estimate} and \eqref{eq:quantitative multiplier estimates} below that one can obtain quantitative bounds in each of the implications between \eqref{item:polynomial alpha-boundedness} and \eqref{item:polynomial resolvent multipliers}.

\begin{theorem}[Characterization of polynomial stability]\label{thm:abstract polynomial stability}
Let $-A$ be the generator of a $C_{0}$-semigroup $(T(t))_{t\geq 0}$ on a Banach space $X$, and assume that $A$ has resolvent growth $(\alpha,\beta)$ for some $\alpha,\beta\in[0,\infty)$. Let $n\in\N_{0}$ and let $Y$ be an $(A,n)$-admissible space.
Then the following statements are equivalent:
\begin{enumerate}[$(1)$]
\item\label{item:polynomial alpha-boundedness}
$\displaystyle \sup_{t\geq0}\|t^n T(t)\|_{\La(Y,X)}<\infty$;
\item\label{item:polynomial resolvent multipliers} there exist $\psi\in C_{c}^{\infty}(\R)$, $p\in[1,\infty)$ and $q\in[p,\infty]$ such that
\begin{align*}
\psi(\cdot)R(\ui\cdot,A)^{k} & \in \Ma_{1,\infty}(\R;\La(Y,X)),
\\ (1-\psi(\cdot))R(\ui\cdot,A)^{k}& \in \Ma_{p,q}(\R;\La(Y,X))
\end{align*}
for all $k\in\{n-1,n,n+1\}\cap\N$.
\end{enumerate}
Moreover, if \eqref{item:polynomial alpha-boundedness} or \eqref{item:polynomial resolvent multipliers} holds then $R(\ui\cdot,A)^{k}\in \Ma_{p,q}(\R;\La(Y,X))$ for:
\begin{enumerate}[(i)]
\item\label{item:Ra} $n\geq 2$, $k\in\{1,\ldots,n-1\}$ and $1\leq p\leq q\leq\infty$;
\item\label{item:Rb} $k=n\geq1$ and $1\leq p<q\leq \infty$;
\item\label{item:Rc} $k=n+1$, $p=1$ and $q = \infty$.
\end{enumerate}
\end{theorem}
\begin{proof}
\eqref{item:polynomial resolvent multipliers}  $\Rightarrow$ \eqref{item:polynomial alpha-boundedness}:
Let $\omega,M_{\w}\geq1$ be such that $\|T(t)\|_{\La(X)} \leq M_{\omega} \ue^{t(\omega-1)}$ for all $t\geq0$, and set
\[
m(\xi):=n!(\ui\xi+A)^{-n}(I_{X}+\w (\ui\xi+A)^{-1})\in\La(Y,X)\qquad  (\xi\in\R\setminus\{0\}).
\]
Since $(\ui\cdot+A)^{-1}=-R(-\ui\cdot,A)$, it follows from Proposition \ref{prop:q-boundedness implies infty-boundedness} that
\begin{equation*}
T_{m}:L^{p}(\R;Y)\cap L^{1}(\R;Y)\to L^{\infty}(\R;X)
\end{equation*}
is bounded with
\begin{align}\label{eq:boundedness multiplier from intersection n larger than one estimate}
\|T_m\|\leq 2M n! (C_{n} + \omega C_{n+1}).
\end{align}
Here $M:=\sup_{t\in [0,2]} \|T(t)\|_{\calL(X)}$, $C_{k}$ is as in Proposition \ref{prop:q-boundedness implies infty-boundedness} for $k\in \N$, and $C_0 := \|I_{Y}\|_{\La(Y,X)}$. Now let $Y_0\subseteq Y$ be as in Definition \ref{def:admissable spaces} and fix $x\in Y_{0}$. Lemma \ref{lem:standardfact} yields
\begin{equation}\label{eq:Fouriertne}
\F [t\mapsto t^{n}T(t)x](\cdot) = n! (\ui \cdot+A)^{-n-1} x.
\end{equation}
Set $f(t):=\ue^{-\w t}T(t)x$ for $t\geq0$, and $f\equiv 0$ on $(-\infty,0)$. Then
\begin{equation}\label{eq:integrability from Y to X}
\|f(t)\|_Y\leq \|\ue^{-\w t}T(t)\|_{\La(Y)} \|x\|_{Y}\leq C_T \|\ue^{-\w t}T(t)\|_{\La(X)} \|x\|_Y\qquad(t\in[0,\infty)).
\end{equation}
Hence $f\in L^{1}(\R;Y)\cap L^{\infty}(\R;Y)$ and $\|f\|_{L^{r}(\R;Y)}\leq C_T M_{\omega} \|x\|_Y$ for all $r\in [1, \infty]$.
By Lemma \ref{lem:standardfact}, $\widehat{f}(\cdot) = (w+\ui\cdot+A)^{-1}x$. Therefore, by the resolvent identity,
\[
m(\xi) \widehat{f}(\xi) = n! (\ui \xi+A)^{-n-1} x\qquad(\xi\in\R\setminus\{0\}).
\]
Combining \eqref{eq:Fouriertne} and \eqref{eq:integrability from Y to X} with
\eqref{eq:boundedness multiplier from intersection n larger than one estimate} yields
\begin{equation}\label{eq:quantitative infty estimate}
\sup_{t\geq0}\|t^{n}T(t)x\|_X \leq \|T_m\| \big( \|f\|_{L^p(\R;Y)} + |f\|_{L^1(\R;Y)}\big) \leq C  \|x\|_Y,
\end{equation}
where $C=4M n! C_T M_{\omega}  (C_{n} + \omega C_{n+1})$. The required result now follows since $Y_0\subseteq Y$ is dense.

\medskip

\eqref{item:polynomial alpha-boundedness} $\Rightarrow$ \eqref{item:polynomial resolvent multipliers}: Set $K_n := \sup_{t\geq0}\|t^{n}T(t)x\|_X$ and let $Y_0\subseteq Y$ be as in Definition \ref{def:admissable spaces}.
Let $f\in \dot{\Sw}(\R)\otimes Y_0$ and set $S_{k}(f)(s):=\int_{0}^{\infty}t^{k}T(t)f(s-t)\ud t$ for $s\in\R$ and $k\in\{0, 1, \ldots, n\}$.
Lemma \ref{lem:standardfact} yields
\begin{equation}\label{eq:boundedness Sk}
S_{k}(f)=k!\F^{-1}((\ui\cdot+A)^{-k-1}\widehat{f}(\cdot))=k!\,T_{(\ui\cdot+A)^{-k-1}}(f).
\end{equation}
Now, for $n\geq 2$, $k\in \{0, \ldots, n-2\}$ and $r\in[1,\infty]$,
\[
\big\|[t\mapsto \|t^{k}T(t)]\|\big\|_{L^r([0,\infty);\La(Y,X))} \ud t \leq M + K_n \|[t\mapsto t^{-2}]\|_{L^r(1, \infty)} \leq  M + K_n.
\]
Similarly, for $n\geq 1$ and $r\in(1,\infty]$,
\[
\big\|[t\mapsto \|t^{n-1}T(t)]\|\big\|_{L^r([0,\infty);\La(Y,X))} \leq M + \frac{K_n}{(r-1)^{1/r}}.
\]
By combining these estimates with \eqref{eq:boundedness Sk} and with Young's inequality for operator-valued kernels in \cite[Proposition 1.3.5]{ArBaHiNe11} one obtains, for $p\in[1,\infty)$ and $q\in[p,\infty]$,
\begin{equation}\label{eq:quantitative multiplier estimates}
\begin{aligned}
\|R(\ui\cdot,A)^{k}\|_{\Ma_{p,q}(\R;\La(Y,X))} & \leq \tfrac{M + K^n}{(k-1)!}\qquad (n\geq 2, k\in \{1, \ldots, n-1\}),
\\ \|R(\ui\cdot,A)^{n}\|_{\Ma_{p,q}(\R;\La(Y,X))} & \leq\tfrac{M+(r-1)^{-1/r}K_{n}}{(n-1)!} \qquad (n\geq 1, p<q),
\\ \|R(\ui\cdot,A)^{n+1}\|_{\Ma_{1,\infty}(\R;\La(Y,X))} &\leq \tfrac{K_n}{n!}.
\end{aligned}
\end{equation}
Now \eqref{eq:boundedness Sk} and \eqref{eq:quantitative multiplier estimates} yield statements \eqref{item:Ra}-\eqref{item:Rc} for $(\ui\cdot+A)^{-1}$, and by reflection these statements hold for $R(\ui\cdot,A)$ as well. Finally, for \eqref{item:polynomial resolvent multipliers} let $\psi\in C_{c}^{\infty}(\R)$. Then Young's inequality and \eqref{eq:quantitative multiplier estimates} yield $\psi(\cdot) R(\ui\cdot,A)^{k} \in\Ma_{1,\infty}(\R;\La(Y,X))$ for all $k\in\{1,\ldots, n+1\}$, and one obtains \eqref{eq:quantitative multiplier estimates} for $\psi(\cdot)R(\ui\cdot,A)$ with an additional multiplicative factor $\|\F^{-1}(\psi)\|_{L^{1}(\R)}$. Similarly, \eqref{eq:quantitative multiplier estimates} holds with an additional multiplicative factor $\|\F^{-1}(1-\psi)\|_{L^{1}(\R)}$ upon replacing $R(\ui\cdot,A)$ by $(1-\psi(\cdot)) R(\ui\cdot,A)$.
\end{proof}

The assumption in Theorem \ref{thm:abstract polynomial stability} that $A$ has resolvent growth $(\alpha,\beta)$ for some $\alpha,\beta\in[0,\infty)$ is only made to ensure that $T_{R(i\cdot,A)}$ is well-defined, and the specific choice of $\alpha$ and $\beta$ is irrelevant here. Inspection of the proof of Theorem \ref{thm:abstract polynomial stability} also shows that one could assume in \eqref{item:polynomial resolvent multipliers} that for each $k\in \{n-1, n, n+1\}\cap \N$ there exist $p_k,q_k\in [1, \infty]$ such that
\begin{align*}
(1-\psi(\cdot))R(\ui\cdot,A)^{k}& \in \Ma_{p_k,q_k}(\R;\La(Y,X)).
\end{align*}
However, we will not need this generality in the remainder. As was already mentioned in Remark \ref{rem:other multiplier conditions at zero}, the assumption
\begin{align*}
\psi(\cdot)R(\ui\cdot,A)^{k}& \in \Ma_{1,\infty}(\R;\La(Y,X))
\end{align*}
in \eqref{item:polynomial resolvent multipliers} is the most general $(L^{p},L^{q})$ Fourier multiplier condition for $\psi(\cdot)R(\ui\cdot,A)^{k}$.

\begin{remark}\label{rem:p<q necessary}
The theory of $(L^{p},L^{p})$ Fourier multipliers alone cannot yield a characterization of polynomial stability as in Theorem \ref{thm:abstract polynomial stability},
and in general it is necessary to also consider the case where $p<q$ in condition \eqref{item:polynomial resolvent multipliers}.
To see this, consider a uniformly bounded $C_{0}$-semigroup $(T(t))_{t\geq 0}\subseteq\La(X)$ with generator $-A$ such that $\overline{\C_{-}}\subseteq\rho(A)$ but $A$ is not of type $(0,0)$. Let $n=0$ and $Y=X$.
Then $R(\ui\cdot,A)\notin\Ma_{p,p}(\R;\La(X))$ for each $p\in[1,\infty)$ since $\sup\{\|R(\ui\xi,A)\|_{\La(X)}\mid \xi\in\R\setminus\{0\}\}=\infty$.
Nonetheless, \eqref{item:polynomial alpha-boundedness} holds since $(T(t))_{t\geq0}$ is uniformly bounded, and $R(\ui\cdot,A)\in \Ma_{1,\infty}(\R;\La(X))$. Indeed,
\[
\F^{-1}(R(\ui\cdot,A)\widehat{f}(\cdot))(t)=\int_{0}^{\infty}T(t-s)f(s)\ud s\qquad (t\in\R)
\]
defines an element of $L^{\infty}(\R;X)$ for each $f\in \Sw(\R;X)$.
\end{remark}

A variation of the proof of Theorem \ref{thm:abstract polynomial stability} yields the following result, which will also be used in Section \ref{sec:exponential stability}. In particular, it provides a simple condition for powers of the resolvent to be Fourier multipliers.

\begin{proposition}\label{prop:abstract polynomial stability}
Let $-A$ be the generator of a $C_{0}$-semigroup $(T(t))_{t\geq 0}$ on a Banach space $X$, and suppose that $\overline{\C_{-}}\setminus\{0\}\subseteq\rho(A)$. Let $q\in[1,\infty)$, $n\in\N_{0}$ and let $Y$ be an $(A,n)$-admissible space.
Then the following statements are equivalent:
\begin{enumerate}[$(1)$]
\item\label{item:polynomial alpha-boundedness prop}
there exists a constant $C\in[0,\infty)$ such that $[t\mapsto t^{n}T(t)x]\in L^{q}([0,\infty);X)$ for all $x\in Y$, and
\[
\|[t\mapsto t^{n}T(t)x]\|_{L^{q}([0,\infty);X)}\leq C\|x\|_{Y}\qquad (x\in Y);
\]
\item\label{item:polynomial integrability prop}
for each $k\in\{n,n+1\}\cap \N$ one has $R(\ui\cdot,A)^{k}\in \Ma_{1,q}(\R;\La(Y,X))$;

\item\label{item:polynomial resolvent multipliers prop} there exist $\psi\in C_{c}^{\infty}(\R)$ and $p\in[1,q]$ such that
\[
\psi(\cdot)R(\ui\cdot,A)^{k}\in\Ma_{1,q}(\R;\La(Y,X))\text{ and } (1-\psi(\cdot))R(\ui\cdot,A)^{k}\in\Ma_{p,q}(\R;\La(Y,X))
\]
for $k\in\{n,n+1\}\cap \N$.
\end{enumerate}
\end{proposition}
\begin{proof}
\eqref{item:polynomial integrability prop} $\Rightarrow$ \eqref{item:polynomial resolvent multipliers prop} is trivial. For
\eqref{item:polynomial resolvent multipliers prop}  $\Rightarrow$ \eqref{item:polynomial alpha-boundedness prop} one proceeds in an almost identical manner as in the proof of Theorem \ref{thm:abstract polynomial stability}, except that now it is not necessary to appeal to Proposition \ref{prop:q-boundedness implies infty-boundedness}.

\eqref{item:polynomial alpha-boundedness prop} $\Rightarrow$ \eqref{item:polynomial integrability prop}:
Let $Y_0\subseteq Y$ be as in Definition \ref{def:admissable spaces}. Then $[t\mapsto t^{k} T(t)x]\in L^{q}(\R_+;X)$ for all $k\in\{0,\ldots, n\}$ and $x\in Y_0$. Hence, for $f\in L^{1}(\R)\otimes Y_0$, Minkowski's inequality yields
\begin{align*}
&\Big(\int_{\R}  \Big\|\int_{\R} (t-s)^n T(t-s) f(s) \ud s\Big\|^q \ud t\Big)^{1/q} \\
&\leq \int_{\R} \Big(\int_{s}^\infty \|(t-s)^n  T(t-s) f(s)\|^q \ud t \Big)^{1/q} \ud s \leq C\int_{\R} \|f(s)\|  \ud s.
\end{align*}
Now the proof is concluded using Lemma \ref{lem:standardfact}.
\end{proof}

\subsection{Results under Fourier type assumptions\label{subsec:Fourier}}

Here we apply Theorem \ref{thm:abstract polynomial stability} to obtain polynomial stability results under assumptions on the Fourier type of the underlying space. The following theorem coincides with Proposition \ref{prop:general semigroups} for $p=1$. In the case where $\alpha=0$ it was already stated in \cite{BaEnPrSc06} that an improvement of Proposition \ref{prop:general semigroups} might be possible using ideas from \cite[\S4.2]{vanNeerven96a}, but no details are given there.

\begin{theorem}\label{thm:polynomial decay Fourier type}
Let $\alpha,\beta\in[0,\infty)$ and let $A$ be an injective sectorial operator with resolvent growth $(\alpha,\beta)$ on a Banach space $X$ with Fourier type $p\in[1,2]$.
Let $r\in[1,\infty]$ be such that $\frac{1}{r}=\frac{1}{p}-\frac{1}{p'}$, and let $\sigma,\tau\in[0,\infty)$ be such that $\sigma>\alpha-1$ and $\tau>\beta+\frac{1}{r}$.
Then for each $\rho\in[0,\min(\frac{\sigma+1}{\alpha}-1,\frac{\tau-r^{-1}}{\beta}-1))$ there exists a $C_{\rho}\in[0,\infty)$ such that
\begin{equation}\label{eq:polynomial decay Fourier type main}
\|T(t)\|_{\La(X^{\sigma}_{\tau},X)}\leq C_{\rho}t^{-\rho}\qquad (t\in[1,\infty)).
\end{equation}
If $p=2$ then \eqref{eq:polynomial decay Fourier type main} also holds for $\tau\geq \beta$ and $\rho\in[0,\infty)$ with $\rho<\frac{\sigma+1}{\alpha}-1$ and $\rho\leq \frac{\tau}{\beta}-1$.
\end{theorem}
\begin{proof}
We prove the following equivalent statement: for all $s\geq0$ and $\delta,\veps>0$ there exists a constant $C_{s,\delta,\veps}\geq0$ such that
\begin{equation}\label{eq:polynomial decay Fourier type}
\|T(t)\|_{\La(X^{\mu}_{\nu},X)}\leq C_{s,\delta,\veps}t^{-s}  \qquad (t\in [1, \infty)),
\end{equation}
where $\mu= \max((s+1) \alpha -1 + \delta,0)$, $\nu = (s+1)\beta+\frac{1}{r}+\veps$ for $p\in[1,2)$, and $\nu=(s+1)\beta$ for $p=2$. By Lemma \ref{lem:interpolation} it suffices to consider $n:=s\in\N_{0}$,
and the case where $p=1$ follows from Proposition \ref{prop:general semigroups}. For $p\in (1, 2)$ set $\beta_{0} := \frac{1}{r}+\veps$, and for $p=2$ we let $\beta_0=0$. We may assume that $\beta_{0}\in [0,1)$.

By Proposition \ref{prop:app-equivalence growth bounds} and because $R(\ui\xi,A)$ commutes with $A^{\alpha}(1+A)^{-\alpha-\beta}$ for all $\xi\in\R\setminus\{0\}$, one has
\begin{equation}\label{eq:boundedness of resolvents}
\sup\{\|R(\ui\xi,A)^{k}\|_{\La(X^{n\alpha}_{n\beta},X)}\mid \xi\in\R\setminus\{0\}\}<\infty\qquad (k\in\{1,\ldots,n\}).
\end{equation}
Now, the part $\widetilde{A}$ of $A$ in $X^{n\alpha}_{n\beta}$ satisfies the conditions of Proposition \ref{prop:app-equivalence growth bounds} and Corollary \ref{cor:app-growth and fractional domains},
and $R(\ui\xi,\widetilde{A})=R(\ui\xi,A)\restriction_{X^{n\alpha}_{n\beta}}$ for all $\xi\in\R\setminus\{0\}$. Hence
\begin{equation}\label{eq:boundedness of resolvent on smaller space}
\Big\{\frac{\abs{\xi}^{1-\delta}}{(1+\abs{\xi})^{1-\delta-\beta_{0}}} R(\ui\xi,A)\Big| \xi\in \R\setminus\{0\}\Big\}\subseteq  \La(X^{\mu}_{\nu},X^{n\alpha}_{n\beta})
\end{equation}
is uniformly bounded. Let $k\in\{1,\ldots, n+1\}$. Then \eqref{eq:boundedness of resolvents} and \eqref{eq:boundedness of resolvent on smaller space} show that
\begin{equation}\label{eq:boundedness of power of resolvent}
\Big\{\frac{\abs{\xi}^{1-\delta}}{(1+\abs{\xi})^{1-\delta-\beta_{0}}} R(\ui\xi,A)^{k}\Big| \xi\in \R\setminus\{0\}\Big\}\subseteq\La(X^{\mu}_{\nu},X)
\end{equation}
is uniformly bounded. Let $\psi\in C_{c}^{\infty}(\R)$ be such that $\psi\equiv 1$ on $[-1,1]$. Since $\delta>0$, it follows from \eqref{eq:boundedness of power of resolvent}
and Proposition \ref{prop:Lp-Lq multipliers Fourier type} that
\[
\psi(\cdot)R(\ui\cdot,A)^{k}\in L^{1}(\R;\La(X^{\mu}_{\nu},X))\subseteq\Ma_{1,\infty}(\R;\La(X^{\mu}_{\nu},X)).
\]
Another application of \eqref{eq:boundedness of power of resolvent} yields $\|(1-\psi(\cdot))R(\ui\cdot,A)^{k}\|_{\La(X^{\mu}_{\nu},X)}\in L^{r}(\R)$.
Note that $X^{\mu}_{\nu}$ has Fourier type $p$, since $X^{\mu}_{\nu}$ is isomorphic to $X$. Hence Proposition \ref{prop:Lp-Lq multipliers Fourier type} yields
\[
(1-\psi(\cdot))R(\ui\cdot,A)^{k}\in \Ma_{p,p'}(\R;\La(X^{\mu}_{\nu},X)).
\]
Now Theorem \ref{thm:abstract polynomial stability} concludes the proof.
\end{proof}

\begin{remark}\label{rem:dependence constant Fourier type}
One can show that the constant $C_{\rho}$ in \eqref{eq:polynomial decay Fourier type main} depends only on the following variables: $\alpha$, $\beta$, $\sigma$, $\tau$, $\rho$, $\F_{p,X}$, the sectoriality constant $M(A)$ from \eqref{eq:sectoriality constant},
\[
M_{\alpha,\beta}:=\sup\Big\{\frac{\abs{\xi}^{\alpha}}{\abs{1+\xi}^{\alpha+\beta}}\|R(\ui\xi,A)\|_{\La(X)}\Big| \xi\in\ui\R\setminus\{0\}\Big\}
\]
and the semigroup growth constants $M$, $\w$ and $M_{\w}$ which appear in \eqref{eq:quantitative infty estimate}.
\end{remark}

It is an open question whether \eqref{eq:polynomial decay Fourier type main} also holds for $\rho=\min(\frac{\sigma+1}{\alpha}-1,\frac{\tau-r^{-1}}{\beta}-1)$ if $\alpha+\beta>0$.

A Hilbert space has Fourier type $2$ by Plancherel's identity. Hence we may distill from Theorem \ref{thm:polynomial decay Fourier type} the following important corollary, which in particular implies Theorem \ref{thm:polynomial decay HilbertIntro}.
It follows from Example \ref{ex:optimalitybeta} and Remark \ref{rem:analyticsemigroup} that, up to $\veps$ loss, the polynomial rate of decay in Corollary \ref{cor:polynomial decay Hilbert space} is optimal for $\alpha=0$ and $\tau=\beta\in[0,\infty)$, and for $\alpha=1$ and $\beta=0$. We do not know whether the rate of decay is also optimal for other values of $\alpha$, $\beta$, $\sigma$ and $\tau$.

\begin{corollary}\label{cor:polynomial decay Hilbert space}
Let $\alpha,\beta\in[0,\infty)$ and let $A$ be an injective sectorial operator with resolvent growth $(\alpha,\beta)$ on a Hilbert space.
Let $\sigma,\tau\in[0,\infty)$ be such that $\sigma>\alpha-1$ and $\tau\geq\beta$. Then for each $\rho\in[0,\infty)$ such that $\rho<\frac{\sigma+1}{\alpha}-1$ and $\rho\leq \frac{\tau}{\beta}-1$ there exists a $C_{\rho}\in[0,\infty)$ such that
\[
\|T(t)\|_{\La(X^{\sigma}_{\tau},X)}\leq C_{\rho}t^{-\rho}\qquad (t\in[1,\infty)).
\]
\end{corollary}

\begin{remark}\label{rem:comparison Fourier type and growth on X}
Corollary \ref{cor:growth of semigroup on X} yields a faster decay rate than Theorem \ref{thm:polynomial decay Fourier type} when $\|T(t)\|_{\La(X)}$ grows slowly as $t\to\infty$. More precisely, with notation as in Theorem \ref{thm:polynomial decay Fourier type}, let $\mu_{0}\in[0,\infty)$ be such that
\[
\min\big(\frac{\sigma}{\alpha},\frac{\tau}{\beta}\big)-\mu_{0}=\min\big(\frac{\sigma+1}{\alpha}-1,\frac{\tau-r^{-1}}{\beta}-1\big).
\]
If there exists a $\mu<\mu_{0}$ such that $\limsup_{t\to\infty}t^{-\mu}\|T(t)\|_{\La(X)}\!<\infty$ then Corollary \ref{cor:growth of semigroup on X} yields a sharper decay rate than Theorem \ref{thm:polynomial decay Fourier type}, namely
\[
\|T(t)\|_{\La(X^{\sigma}_{\tau},X)}\lesssim t^{-\rho}\qquad (t\in[1,\infty))
\]
for each $\rho<\min(\frac{\sigma}{\alpha},\frac{\tau}{\beta})-\mu$. Otherwise Theorem \ref{thm:polynomial decay Fourier type} yields at least as sharp a decay rate as Corollary \ref{cor:growth of semigroup on X}. In particular, on Hilbert spaces Corollary \ref{cor:polynomial decay Hilbert space} yields faster decay than Corollary \ref{cor:growth of semigroup on X} if $\alpha=0$ and $\|T(\cdot)\|_{\La(X)}$ grows at least linearly. Note also that in many cases \eqref{eq:better rates using literature} below yields a faster decay rate than Corollary \ref{cor:growth of semigroup on X}.
\end{remark}

\subsection{Results under type and cotype assumptions\label{subsec:cotype}}

Here we consider polynomial decay rates under type and cotype assumptions on the underlying space.

The following result also holds for $q=\infty$. However, in this case Proposition \ref{prop:general semigroups} yields a more general statement, since each Banach space has type $p=1$ and cotype $q=\infty$ and because a Banach space with nontrivial type also has finite cotype.

\begin{theorem}\label{thm:polynomial decay type}
Let $\alpha,\beta\in[0,\infty)$ and let $A$ be an injective sectorial operator with $R$-resolvent growth $(\alpha,\beta)$ on a Banach space $X$ with type $p\in[1,2]$ and cotype $q\in[2,\infty)$. Let $r\in[1,\infty]$ be such that $\frac{1}{r}=\frac{1}{p}-\frac{1}{q}$, and let $\sigma,\tau\in[0,\infty)$ be such that $\sigma>\alpha-1$ and $\tau>\beta+\frac{1}{r}$. Then for each $\rho<\min(\frac{\sigma+1}{\alpha}-1,\frac{\tau-r^{-1}}{\beta}-1)$ there exists a $C_{\rho}\in[0,\infty)$ such that
\[
\|T(t)\|_{\La(X^{\sigma}_{\tau},X)}\leq C_{\rho}t^{-\rho}\qquad (t\in[1,\infty)).
\]
If $p=q=2$ then \eqref{eq:polynomial decay Fourier type main} also holds for $\tau\geq\beta$ and $\rho\in[0,\infty)$ with $\rho<\frac{\sigma+1}{\alpha}-1$ and $\rho\leq \frac{\tau}{\beta}-1$.
\end{theorem}
\begin{proof}
The proof is similar to that of Theorem \ref{thm:polynomial decay Fourier type}. The case where $p=q=2$ is already contained in Corollary \ref{cor:polynomial decay Hilbert space}, since each Banach space with type $2$ and cotype $2$ is isomorphic to a Hilbert space, and because every uniformly bounded collection on a Hilbert space is $R$-bounded. So we may assume that $r\in(1,\infty)$ and derive \eqref{eq:polynomial decay Fourier type} for $n:=s\in\N_{0}$. Set $\beta_{0} := \frac{1}{r}+\veps$ and let $k\in\{1,\ldots, n+1\}$. We may suppose that $\beta_{0}\in(0,1)$. As in the proof of Theorem \ref{thm:polynomial decay Fourier type}, using Proposition \ref{prop:app-equivalence growth bounds} and Corollary \ref{cor:app-growth and fractional domains}, one sees that
\[
\{\abs{\xi}^{1-\delta}R(\ui\xi,A)^{k}\mid \xi\in\R\setminus\{0\},\abs{\xi}\leq 1\}\subseteq\La(X^{\mu},X)
\]
is uniformly bounded and that
\begin{equation}\label{eq:polynomial decay type 1}
\{\abs{\xi}^{\beta_{0}}R(\ui\xi,A)^{k}\mid \xi\in\R\setminus\{0\},\abs{\xi}\geq 1\}\subseteq\La(X_{\nu},X)
\end{equation}
is $R$-bounded. Now let $\psi\in C^{\infty}_{c}(\R)$ be such that $\psi \equiv 1$ on $[-1/2, 1/2]$ and such that $\supp(\psi)\subseteq[-1,1]$. Then Proposition \ref{prop:Lp-Lq multipliers Fourier type} shows that $\psi(\cdot)R(\ui\cdot,A)^{k}\in\Ma_{1,\infty}(\La(X^{\mu},X))$, and $(1-\psi(\cdot)R(\ui\cdot,A)^{k}\in \mathcal{M}_{p,q}(\La(X_{\nu},X))$ by the first statement in Proposition \ref{prop:Lp-Lq multipliers type}. Theorem \ref{thm:abstract polynomial stability} concludes the proof.
\end{proof}

A similar dependence on the underlying parameters as in Remark \ref{rem:dependence constant Fourier type} holds for the constant $C_{\rho}$ in Theorem \ref{thm:polynomial decay type}.

Using the second statement in Proposition \ref{prop:Lp-Lq multipliers type} we obtain the following improvement of Theorem \ref{thm:polynomial decay type} on Banach lattices, which allows one to deal with the limit case in the fractional domain exponent.

\begin{theorem}\label{thm:polynomial decay convex}
Let $\alpha,\beta\in[0,\infty)$ and let $A$ be an injective sectorial operator with $R$-resolvent growth $(\alpha,\beta)$ on a Banach lattice $X$ which is $p$-convex and $q$-concave for $p\in[1,2]$ and $q\in[2,\infty)$. Let $r\in(1,\infty]$ be such that $\frac{1}{r}=\frac{1}{p}-\frac{1}{q}$, and let $\sigma,\tau\in[0,\infty)$ be such that $\sigma>\alpha-1$ and $\tau\geq \beta+\frac{1}{r}$. Then for each $\rho\in[0,\infty)$ such that $\rho<\frac{\sigma+1}{\alpha}-1$ and $\rho\leq \frac{\tau-r^{-1}}{\beta}-1$ there exists a $C_{\rho}\in[0,\infty)$ such that
\[
\|T(t)\|_{\La(X^{\sigma}_{\tau},X)}\leq C_{\rho}t^{-\rho}\qquad (t\in[1,\infty)).
\]
\end{theorem}

We do not know whether the $R$-boundedness assumption in Theorems \ref{thm:polynomial decay type} and \ref{thm:polynomial decay convex} is necessary. This question is relevant even in the case where $\alpha=\beta=0$, cf.~the remark following Corollary \ref{cor:R-boundedness for free}.

\begin{remark}\label{rem:comparison type and Fourier type}
Each Banach space $X$ with Fourier type $p\in[1,2]$ has type $p$ and cotype $p'$, but the converse does not hold in general. In particular, if $X=L^{u}(\Omega)$ for $u\in[1,\infty)$ and for some measure space $\Omega$, then $X$ has Fourier type $\wt{p}=\min(u,u')$, type $p=\min(u,2)$ and cotype $q=\max(u,2)$. In this case the parameter $\frac{1}{r}$ in Theorems \ref{thm:polynomial decay type} and \ref{thm:polynomial decay convex} is strictly smaller than in Theorem \ref{thm:polynomial decay Fourier type} for $u\in[1,\infty)\setminus\{2\}$. However, the $R$-boundedness assumption on the resolvent of $A$ is in general stronger than the assumption in Theorem \ref{thm:polynomial decay Fourier type}.

We suspect that the $R$-boundedness condition in Theorems \ref{thm:polynomial decay type} and \ref{thm:polynomial decay convex} can be removed at the cost of a larger parameter $\frac{1}{r}$. For $\alpha=\beta=0$ this is indeed the case, with $\frac{1}{r}=2(\frac{1}{p}-\frac{1}{q})$, as is shown in Corollary \ref{cor:R-boundedness for free}.
\end{remark}

\subsection{Results for asymptotically analytic semigroups}

Here we consider polynomial stability for the asymptotically analytic semigroups from \cite{Blake99}. Define the \emph{non-analytic growth bound} $\zeta(T)$ of a $C_{0}$-semigroup $(T(t))_{t\geq 0}$ on a Banach space $X$ as
\[
\zeta (T)  := \inf \Big\{ \omega\in\R \Big|  \sup_{t>0}e^{-\omega t}\| T(t)-S(t)\|<\infty\;\mbox{for some $S\in\mathcal{H}(\La(X))$}\Big\},\]
where $\mathcal{H}(\La(X))$ is the set of $S\colon(0,\infty)\to\La(X)$ having an exponentially bounded analytic extension to some sector containing $(0,\infty)$.
One says that $(T(t))_{t\geq 0}$ is \emph{asymptotically analytic} if $\zeta(T)<0$. In this case $s_{0}^{\infty}(-A)<0$, where $s_{0}^{\infty}(-A)$  is the infimum over all $\w\in\R$ for which there exists an $R>0$ such that
\[
\{\lambda\in\C\mid \Real(\lambda)\geq\w,\abs{\Imag(\lambda)}\geq R\}\subseteq\rho(-A)
\]
and
\[
\sup\{\|(\lambda+A)^{-1}\|_{\La(X)}\mid \Real(\lambda)\geq\w,\abs{\Imag(\lambda)}\geq R\}<\infty.
\]
The converse implication holds if $X$ is a Hilbert space. More generally, it was shown in \cite[Theorem 3.6]{Batty-Srivastava03} that $\zeta(T)<0$ if and only if $s_{0}^{\infty}(-A)<0$ and there exist $R>0$ and $\psi\in C^{\infty}_{c}(\R)$ such that $\ui(\R\setminus [-R,R])\subseteq\rho(A)$, $\psi\equiv 1$ on $[-R,R]$ and
\[
(1-\psi(\cdot))R(\ui\cdot,A)\in \Ma_{p,p}(\R;\La(X))
\]
for some (in which case it holds for all) $p\in[1,\infty)$.

 Note that if $(T(t))_{t\geq 0}$ is analytic, and in particular if $A$ is bounded, then trivially $\zeta(T)=-\infty$. More generally, if $(T(t))_{t\geq 0}$ is eventually differentiable then $\zeta(T)=-\infty$. For these facts and for more on the non-analytic growth bound see \cite{Blake99,BaBlSr03,Batty-Srivastava03}.

\begin{theorem}\label{thm:asymptotically analytic semigroups}
Let $\alpha\in[0,\infty)$ and let $A$ be an injective sectorial operator with resolvent growth $(\alpha,0)$ on a Banach space $X$. Suppose that $(T(t))_{t\geq 0}$ is asymptotically analytic, and let $\sigma\in[0,\infty)$ be such that $\sigma>\alpha-1$.
Then for each $\rho\in[0,\frac{\sigma+1}{\alpha}-1)$ there exists a $C_{\rho}\in[0,\infty)$ such that
\[
\|T(t)\|_{\La(X^{\sigma},X)}\leq C_{\rho}t^{-\rho}\qquad (t\in[1,\infty)).
\]
\end{theorem}
\begin{proof}
It suffices to obtain \eqref{eq:polynomial decay Fourier type} with $\mu= \max((n+1) \alpha -1 + \delta,0)$ and $\nu=0$ for $n\in\N_{0}$.
There exist $R\in(0,\infty)$, $\psi\in C^{\infty}_{c}(\R)$ and $p\in[1,\infty)$ such that
\begin{equation}\label{eq:multiplier in Mp}
(1-\psi(\cdot))R(\ui\cdot,A)^{k}\in\Ma_{p,p}(\R;\La(X)) \qquad (k\in\{0,\ldots, n+1\}).
\end{equation}
Since the inclusion $X^{\mu}\subseteq X$ is continuous, \eqref{eq:multiplier in Mp} also holds with $\La(X)$ replaced by $\La(X^{\mu},X)$. It follows as in the proof of Theorem \ref{thm:polynomial decay Fourier type} that
\[
\psi(\cdot)R(\ui\cdot,A)^{k}\in L^{1}(\R;\La(X^{\mu},X)\subseteq\Ma_{1,\infty}(\R;\La(X^{\mu},X)) \qquad (k\in\{0,\ldots, n+1\}).
\]
Now Theorem \ref{thm:abstract polynomial stability} yields the required estimate.
\end{proof}

\begin{remark}\label{rem:analyticsemigroup}
An injective sectorial operator $A$ of angle $\ph\in(0,\pi/2)$ has resolvent growth $(1,0)$. The semigroup $(T(t))_{t\geq 0}$ generated by $-A$ is analytic and for any $\sigma\geq 0$ one has
\[
\|T(t)\|_{\La(X^{\sigma},X)}\lesssim t^{-\sigma}\qquad (t\in[1,\infty)).
\]
This follows from \cite[Proposition 2.6.11]{Haase06a}.
This decay rate is optimal for the multiplication semigroup $(T(t))_{t\geq 0}$ on $L^p[0,\infty)$, $p\in[1,\infty)$, given by $T(t)f(s)=\ue^{-ts}f(s)$ for $f\in L^{p}[0,\infty)$ and $t,s\geq0$.
\end{remark}

\subsection{Necessary conditions}

In this subsection we study the necessity of the assumptions in our results.

\subsubsection{Spectral conditions}

The following lemma, an extension of \cite[Proposition 6.4]{BaChTo16}, shows that one can deduce spectral properties of an operator $A$ given uniform decay on suitable subspaces of the associated semigroup. The proof follows that of \cite[Proposition 6.4]{BaChTo16} and uses the Hille--Phillips functional calculus for semigroup generators. For more on this calculus see \cite[Section 3.3]{Haase06a} or \cite[Chapter XV]{Hille-Phillips57}.

\begin{lemma}\label{lem:resolventconv}
Let $-A$ be the generator of a $C_0$-semigroup $(T(t))_{t\geq 0}$ on a Banach space $X$. Suppose that there exist $\alpha,\beta\in\N_{0}$, $\eta\in\rho(-A)$, and a sequence $(t_{n})_{n\in\N}\subseteq[0,\infty)$ such that
\begin{equation}\label{eq:resolventconv}
\lim_{n\to \infty} \|T(t_{n}) A^\alpha (\eta+A)^{-\alpha-\beta}\|_{\La(X)}= 0.
\end{equation}
Then $\overline{\C_{-}}\setminus\{0\}\subseteq\rho(A)$.
\end{lemma}
\begin{proof}
Without loss of generality we may consider $\beta\in\N$ and $\eta>\w_{0}(T)$. Let $t\geq0$ and set $f_{t}(\lambda):=\ue^{-t\lambda}\lambda^{\alpha}(\eta+\lambda)^{-\alpha-\beta}$ for $\lambda\in\C$ with $\Real(\lambda)>-\eta$. Let
\[
k(s) := \left\{\!
           \begin{array}{ll}\!
             \frac{1}{(\alpha+\beta-1)!}\frac{d^\alpha}{ds^\alpha} (s^{\alpha+\beta-1}e^{-\eta s}), & s\in(0,\infty), \\
             \!0, & s\in(-\infty,0].
           \end{array}
         \right.
\]
Then $f_{t}$ is the Laplace transform of $\delta_{t}\ast k$, where $\delta_{t}$ is the Dirac point mass at $t$. Moreover, $f_t(A)$ is defined through the Hille--Phillips functional calculus for $A$ and
\[
f_{t}(A) = T(t) A^{\alpha}(\eta+A)^{-\alpha-\beta}.
\]
By the spectral inclusion theorem for the Hille--Phillips functional calculus in \cite[Theorem 16.3.5]{Hille-Phillips57} one obtains $f_{t}(\sigma(A))\subseteq\sigma(f_{t}(A))$. Let $\lambda\in \sigma(A)\setminus \{0\}$ and $n\in\N$. Then $f_{t_{n}}(\lambda)\in \sigma(f_{t_{n}}(A))$, so \cite[Corollary IV.1.4]{Engel-Nagel00} shows that
\[
e^{-\Real(\lambda) t_{n}} \frac{|\lambda|^\alpha}{|\eta + \lambda|^{\alpha+\beta}}=  |f_{t_{n}}(\lambda)|\leq \|f_{t_{n}}(A)\|= \|T(t_{n}) A^{\alpha}(\eta+A)^{-\alpha-\beta}\|.
\]
This concludes the proof since the right-hand side tends to zero as $n\to \infty$.
\end{proof}

If $\eta+A$ is a sectorial operator in Lemma \ref{lem:resolventconv}, then one may consider $\beta\in[0,\infty)$ in \eqref{eq:resolventconv}.
Similarly, if $A$ is a sectorial operator then one may let $\alpha\in[0,\infty)$.

A similar statement as in the following proposition can be obtained for more general subspaces. It follows from Example \ref{ex:mult} that the conclusion is sharp. 

\begin{proposition}\label{prop:decay implies resolvent bounds2}
Let $A$ be an injective sectorial operator such that $-A$ generates a $C_0$-semigroup $(T(t))_{t\geq 0}$ on $X$. Suppose that there exist $\alpha\in \{0\}\cup[1, \infty)$, $\beta\in[0,\infty)$ and a $g\in L^1(\R_+)$ such that $\|T(t)\|_{\La(X_\beta^{\alpha},X)}\leq g(t)$ for $t\geq 0$.
Then $\overline{\C_{-}}\setminus\{0\}\subseteq\rho(A)$ and
\begin{equation}\label{eq:familyconvrbdd}
\{\lambda^{\alpha}(1+\lambda)^{-\alpha-\beta} (\lambda+A)^{-1}\mid \lambda\in\overline{\C_{+}}\setminus\{0\}\}\subseteq\La(X)
\end{equation}
is $R$-bounded. In particular, $A$ has $R$-resolvent growth $(\alpha, \beta)$. Furthermore, if $\alpha=0$ then also $0\in \rho(A)$.
\end{proposition}
\begin{proof}
Lemma \ref{lem:resolventconv} and the remark following it show that $\overline{\C_{-}}\setminus\{0\}\subseteq\rho(A)$. By assumption, $T(\cdot)A^{\alpha}(1+A)^{-\alpha-\beta}x\in L^{1}(\R_{+};X)$ for all $x\in X$, with
\[
\int_{0}^{\infty}\|T(t)A^{\alpha}(1+A)^{-\alpha-\beta}x\|_{X}\ud t\leq \|g\|_{L^{1}(\R_{+})}\|x\|_{X}.
\]
Moreover, for each $\lambda\in\overline{\C}_{+}$ one has $[t\mapsto e^{-\lambda t}]\in L^{\infty}(\R_{+})$. Set
\[
F(\lambda)x:=\int_{0}^{\infty}e^{-\lambda t}T(t)A^{\alpha}(1+A)^{-\alpha-\beta}x\,\ud t
\]
for $\lambda\in\overline{\C_{+}}$ and $x\in X$. By \cite[Corollary 2.17]{Kunstmann-Weis04}, $\{F(\lambda)\mid \lambda\in\overline{\C_{+}}\setminus\{0\}\}\subseteq\La(X)$ is $R$-bounded. Lemma \ref{lem:standardfact}, applied to the semigroup $(e^{-\lambda t}T(t))_{t\geq0}$ generated by $-(\lambda+A)$, shows that $F(\lambda)=(\lambda+A)^{-1}A^{\alpha}(1+A)^{-\alpha-\beta}$ for $\lambda\in\overline{\C_{+}}\setminus\{0\}$. Now Proposition \ref{prop:app-equivalence growth bounds} \eqref{item:app-equivalence growth bounds 3} implies that
\begin{equation}\label{eq:decay and resolvent bounds}
\{\lambda^{\alpha}(1+\lambda)^{-\alpha-\beta} (\lambda+A)^{-1}\mid \lambda\in\C\setminus\{0\}, |\!\arg(\lambda)|\in[\ph,\pi/2]\}\subseteq\La(X)
\end{equation}
is $R$-bounded for each $\ph\in(0,\pi/2)$. In particular, since $A$ is a sectorial operator, the collection in \eqref{eq:familyconvrbdd} is uniformly bounded. Now a standard argument, considering a convolution with the Poisson kernel (see e.g.~\cite[Proposition 8.5.8]{HyNeVeWe2}), shows that \eqref{eq:familyconvrbdd} is $R$-bounded.

For the second statement suppose that $\alpha=0$. To show that $0\in\rho(A)$ we may consider $\beta\in\N$, since $(1+A)^{-(\lceil\beta\rceil-\beta)}\in\La(X)$. Note that $F(0)\in\La(X,D(A))$, with
\[
AF(0)x=-\lim_{h\downarrow0}\frac{T(h)-I_{X}}{h}F(0)=\lim_{h\downarrow0}\frac{1}{h}\int_{0}^{h}T(t)(1+A)^{-\beta}x\,\ud t=(1+A)^{-\beta}x
\]
for all $x\in X$. Similarly, $F(0)Ay=(1+A)^{-\beta}y$ for $y\in D(A)$. By iteration one obtains that $F(0)\in\La(X,X_{\beta})$. This shows that the part of $A$ in $X_{\beta}$ is invertible, with inverse $F(0)(1+A)^{\beta}|_{X_{\beta}}$. Using the similarity transform $(1+A)^{-\beta}:X\to X_\beta$ one obtains $0\in\rho(A)$, which concludes the proof.
\end{proof}

\subsubsection{Operators which are not sectorial}\label{subsec:non sectorial operators}

In several of the results up to this point we have considered operators $A$ with resolvent growth $(\alpha,\beta)$, for $\alpha,\beta\in[0,\infty)$, which are in addition assumed to be sectorial. Here we discuss which results are still valid when one drops the sectoriality assumption. A complicating factor is then that $A^{\alpha}$ is not well defined through the sectorial functional calculus, and we only consider $\alpha\in \N_{0}$.

Let $A$ be an injective operator, not necessarily sectorial, with resolvent growth $(\alpha,\beta)$ on a Banach space $X$. First note that $\veps+A$ is a sectorial operator for each $\veps>0$, since $-A$ generates a $C_{0}$-semigroup and $\overline{\C}_{-}\subseteq\rho(\veps+A)$. Hence the fractional domains
\[
X_{\beta}=D((1+A)^{-\beta})=D((1-\veps+\veps+A)^{-\beta})
\]
 are well defined via the sectorial functional calculus for $\veps+A$, $\veps\in(0,1)$, and up to norm equivalence they do not depend on the choice of $\veps$.

If $\alpha_{1}=\alpha_{2}\in\N_{0}$ in Lemma \ref{lem:interpolation}, then \eqref{eq:decay on intermediate domains} still holds and \eqref{eq:decay on larger domains} is replaced by
\begin{equation}\label{eq:decay larger domains nonsectorial}
\|T(t)\|_{\La(X^{\lceil \nu\alpha_{1}\!\rceil}_{\nu\beta_{1}},X)}\leq C_{\nu} t^{-\mu\nu}\qquad(t\in[1,\infty)).
\end{equation}
The proofs are identical except that one obtains \eqref{eq:decay larger domains nonsectorial} for $\nu\notin\N$ by applying \eqref{eq:decay on intermediate domains} to the pairs $(\lceil\nu\rceil\alpha_1,
\lfloor\nu\rfloor\beta_1)$ and $(\lceil\nu\rceil\alpha_1,\lceil\nu\rceil\beta_1)$. One can also show that for each $\tau\in[0,\infty)$ there exists a $\sigma\in\N$ such that \eqref{eq:growth of semigroup on X main} holds for all $\rho\in[0,\min(\frac{\sigma}{\alpha},\frac{\tau}{\beta}))$.

Suppose that $X$ has Fourier type $p\in[1,2]$ and let $r\in[1,\infty]$ be such that $\frac{1}{r}=\frac{1}{p}-\frac{1}{p'}$.
Then for all $s\geq0$ and $\veps>0$ there exists a $C_{s,\veps}\geq0$ such that
\begin{equation}\label{eq:polynomial decay Fourier type nonsectorial}
\|T(t)\|_{\La(X^{\mu}_{\nu},X)}\leq C_{s,\veps}t^{-s}  \qquad (t\in [1, \infty)),
\end{equation}
where $\mu= \lfloor (\lceil s\rceil +1)\alpha\rfloor\in\N_{0}$, $\nu = (s+1)\beta+\frac{1}{r}+\veps$ for $p\in[1,2)$, and $\nu=(s+1)\beta$ for $p=2$. Versions of \eqref{eq:polynomial decay Fourier type nonsectorial} in the settings of Theorems \ref{thm:polynomial decay type}, \ref{thm:polynomial decay convex} and \ref{thm:asymptotically analytic semigroups} also hold. In particular, if $(T(t))_{t\geq 0}$ is asymptotically analytic then
\eqref{eq:polynomial decay Fourier type nonsectorial} holds with $\mu:= \lfloor(s+1) \alpha\rfloor$ and $\nu=0$ for each $s\in\N_{0}$.

\subsection{Comparison and examples}\label{subsec:comparison, problems and examples}

In this subsection we compare the decay rates which we have obtained to what can be found in the literature, and we present examples to illustrate our results.

\subsubsection{Comparison}

Let $\alpha,\beta\geq0$ and let $A$ be an injective sectorial operator with resolvent growth $(\alpha,\beta)$ on a Banach space $X$. The decay rates which we have obtained so far are in general not optimal when $(T(t))_{t\geq 0}\subseteq\La(X)$ is uniformly bounded. Indeed, for $\sigma,\tau\geq0$ and $N:=\sup\{\|T(t)\|_{\La(X)}\mid t\in[0,\infty)\}<\infty$ it follows from \cite{Batty-Duyckaerts08,Chill-Seifert16} that there exists a $C_{\rho}\geq0$ such that
\begin{equation}\label{eq:comparison bounded}
\|T(t)\|_{\La(X^{\sigma}_{\tau},X)}\leq C_{\rho}N (1+\log(t))^{\rho}t^{-\rho}\qquad (t\in[1,\infty)),
\end{equation}
where $\rho=\frac{\sigma}{\alpha}$ if $\beta=0$, $\rho=\frac{\tau}{\beta}$ if $\alpha=0$, and $\rho=\min(\sigma,\tau)\cdot\min(\frac{1}{\alpha},\frac{1}{\beta})$ if $\alpha\beta>0$. It was shown in \cite{Borichev-Tomilov10} that \eqref{eq:comparison bounded} is optimal on general Banach spaces if $\alpha=0$, but on Hilbert spaces \eqref{eq:comparison bounded} can be improved to
\begin{equation}\label{eq:comparison bounded Hilbert}
\|T(t)\|_{\La(X^{\sigma}_{\tau},X)}\leq C_{\rho}N^{2}t^{-\rho}\qquad (t\in[1,\infty)),
\end{equation}
cf.~\cite{Borichev-Tomilov10,BaChTo16}. Moreover, \eqref{eq:comparison bounded Hilbert} is optimal, in the sense that for $\sigma,\tau\in\{0,1\}$ \eqref{eq:comparison bounded Hilbert} implies that $A$ has resolvent growth $(\alpha,\beta)$ (see \cite{Batty-Duyckaerts08, BaChTo16}).

For unbounded semigroups \eqref{eq:comparison bounded} and \eqref{eq:comparison bounded Hilbert} do not hold in general.
Indeed, \cite[Example 4.2.9]{vanNeerven96b} gives an example of an operator $A$ with $R$-resolvent growth $(0,0)$ on $X := L^p(1,\infty)\cap L^{p'}(1,\infty)$, $p\in[1,2)$, such that $\|T(\cdot)\|_{\La(X)}$ grows exponentially. Moreover, Example \ref{ex:optimalitybeta} shows that on Hilbert spaces \eqref{eq:comparison bounded Hilbert} can fail for $\alpha=0$ and $\beta>0$, and Corollary \ref{cor:polynomial decay Hilbert space} is optimal for this example when $\tau=\beta$.

Note that \eqref{eq:comparison bounded} need not be optimal for uniformly bounded semigroups when $\alpha\beta>0$, and that Corollary \ref{cor:growth of semigroup on X} yields a sharper decay rate if e.g.~$\alpha=\sigma=1/\veps$ and $\beta=\tau=\veps$ for $\veps\in(0,1)$. On Hilbert spaces one can use \cite[Theorem 4.7]{BaChTo16}, Proposition \ref{prop:app-equivalence growth bounds} and Lemma \ref{lem:interpolation} to let $\rho=\min(\frac{\sigma}{\alpha},\frac{\tau}{\beta})$ in \eqref{eq:comparison bounded Hilbert}, but a similar improvement of \eqref{eq:comparison bounded} on Banach spaces using the methods of \cite{Batty-Duyckaerts08, Chill-Seifert16} is not immediate.

The characterization of polynomial stability in Theorem \ref{thm:abstract polynomial stability} is new even for uniformly bounded semigroups.

A scaling argument can be used to apply \eqref{eq:comparison bounded} to polynomially growing semigroups, leading to sharper decay rates than those in Corollary \ref{cor:growth of semigroup on X}. Suppose $\alpha=0$, $\beta>0$ and that $\|T(t)\|_{\La(X)}\lesssim t^{\mu}$ for all $t\geq1$ and some $\mu\geq0$. For $a>0$ one has
\[
\sup\{\|\ue^{-at}T(t)\|_{\La(X)}\mid t\in[0,\infty)\} \lesssim a^{-\mu}.
\]
Now \eqref{eq:comparison bounded} yields
\[
\|\ue^{-at} T(t)\|_{\La(X_{\tau},X)}\lesssim a^{-\mu} (1+\log(t))^{\tau/\beta}t^{-\tau/\beta}\qquad(t\in[1,\infty)).
\]
For $t\geq1$ set $a := 1/t$. Then
\begin{equation}\label{eq:better rates using literature}
\|T(t)\|_{\calL(X_{\sigma},X)}\lesssim (1+\log(t))^{\tau/\beta}t^{\mu-\tau/\beta},
\end{equation}
which improves the rates from Corollary \ref{cor:growth of semigroup on X}. However, other results in this section yield faster decay rates than \eqref{eq:better rates using literature} for large $\mu$, such as Corollary \ref{cor:polynomial decay Hilbert space} for $\mu\geq 1$.

In this article we make polynomial growth assumptions on the resolvent, whereas in \cite{Batty-Duyckaerts08,BaChTo16,Chill-Seifert16, RoSeSt17} more general resolvent growth is allowed. The scaling argument from above can be used in certain cases to obtain  decay estimates corresponding to more general resolvent growth, but this depends on the growth behavior of the semigroup on $X$. We do not know whether the techniques from this article can be used to obtain nontrivial decay estimates for unbounded semigroups under, for example, exponential or logarithmic growth conditions on the resolvent.

\subsubsection{An exponentially unstable semigroup with polynomial resolvent}

We now apply our theorems to an operator from \cite[Example 4.1]{Wrobel89}, which in turn is a variation of a classical example in stability theory from \cite{Zabczyk75} (see also \cite[Example 1.2.4]{vanNeerven96b}). This example shows that Corollary \ref{cor:polynomial decay Hilbert space} is optimal in the case where $\alpha=0$ and $\tau=\beta$. %We leave open the question whether Corollary \ref{cor:polynomial decay Hilbert space} is also optimal for other values of $\rho$. 
%the value $\tau = \beta>0$ in Corollary \ref{cor:polynomial decay Hilbert space} is sharp in the sense that for any $\tau\in [0,\beta)$ we can construct an operator $A$ such that $A$ has resolvent growth $(0,\beta)$, but the semigroup grows exponentially on $X_{\tau}$. 
%We do not know whether Corollary \ref{cor:polynomial decay Hilbert space} is also optimal for other values of $\tau$.

\begin{example}\label{ex:optimalitybeta}
We show that for all $\beta\in(0,\infty)$ and $\veps\in(0,1)$ there exists an operator $A$ with resolvent growth $(0,\beta)$ on a Hilbert space $X$ such that $\|T(\cdot)\|_{\La(X_{\tau},X)}$ is unbounded for $\tau\in[0,(1-\veps)\beta)$. In fact, $\|T(t)\|_{\La(X_{\tau},X)}$ grows exponentially in $t$ for $\tau\in[0,(1-\veps)\beta)$. By Corollary \ref{cor:polynomial decay Hilbert space} $\|T(\cdot)\|_{\La(X_{\beta},X)}$ is uniformly bounded, and therefore the exponent $\tau$ in Corollary \ref{cor:polynomial decay Hilbert space} is optimal.

It suffices to show that for all $\gamma,\delta\in(0,1)$ there exists an operator $A$ with resolvent growth $\big(0,\frac{\log(1/\gamma)}{\log(1/\delta)}\big)$ on a Hilbert space $X$ such that $\|T(\cdot)\|_{\La(X_{\tau},X)}$ is unbounded for all $\tau\in[0,\frac{1-\gamma}{\log(1/\delta)})$, as follows from the fact that $1-\gamma$ is the first order Taylor approximation of $\log(1/\gamma)$ near $\gamma=1$. Set $\beta_{0}:=\frac{\log(1/\gamma)}{\log(1/\delta)}$, and for $n\in\N$ let the $n\times n$ matrix $B_{n}$ be given by
\[
B_{n}(k,l):= \left\{\!
    \begin{array}{ll}
      \!1 & \text{for }l=k+1, \\
      \!0 & \text{otherwise.}
    \end{array}
  \right.
\]
Let $m(n):= \big\lfloor \frac{\log(n)}{\log(1/\delta)}\big\rfloor\in\N_{0}$ and let $n_{0}\in\N$ be such that $m(n_{0})\geq 2$. Next, let $\displaystyle X = \bigoplus_{n\geq n_{0}} \ell^2_{m(n)}$ be the $\ell^2$ direct sum of the $m(n)$-dimensional $\ell^2_{m(n)}$ spaces for $n\geq n_{0}$, and consider the operator $A := (-\ui n+\gamma-B_{m(n)})_{n\geq n_{0}}$ on $X$. As shown in \cite[Example 4.1]{Wrobel89}, $-A$ generates a $C_0$-semigroup $(T(t))_{t\geq0}\subseteq\La(X)$ such that $\omega_0(T) = 1-\gamma$. We claim that $\overline{\C_{-}}\subseteq \rho(A)$ and that there exists a $C\geq0$ such that
\begin{equation}\label{eq:wrobel}
\|(\eta+\ui\xi+A)^{-1}\|_{\La(X)}\leq C (|\xi|^{\beta_{0}}+1) \qquad(\eta\in[0,\infty),\ \xi\in \R),
\end{equation}
which implies that $A$ has resolvent growth $(0,\beta_{0})$.

To prove the claim let $z:=\eta+i\xi$ and note that $B_{m(n)}^{m(n)} = 0\in \La(\ell^{2}_{m(n)})$ and $\|B_{m(n)}\|_{\La(\ell^{2}_{m(n)})}=1$ for all $n\geq n_{0}$. Hence
\[
\|(z-\ui n+\gamma-B_{m(n)})^{-1}\|_{\La(\ell^{2}_{m(n)})} \leq \sum_{k=0}^{m(n)-1} \frac{\|B_{m(n)}^k\|_{\La(\ell^{2}_{m(n)})}}{|z-\ui n+\gamma|^{k+1}}\leq \sum_{k=0}^{m(n)-1} \frac{1}{|z-\ui n+\gamma|^{k+1}}.
\]
Fix $\xi\in \R$, and let $n_1\in\N$ be such that $n_{1}\geq n_{0}$ and $|n_1-\xi| = \min\{\abs{n-\xi}\mid n\in\N, n\geq n_{0}\}$. Note that $\abs{z-\ui n+\gamma}\geq \gamma$ for all $n\in\N$. Hence for $\xi\geq 0$ and $n\in \{n_{0}, \ldots, n_1+1\}$ one has
\begin{align*}
&\|(z-\ui n+\gamma-B_{m(n)})^{-1}\|_{\La(\ell^{2}_{m(n)})}  \leq \sum_{k=0}^{m(n)-1} \frac{1}{\gamma^{k+1}} = \frac{\gamma^{-m(n)}-1}{1-\gamma}\\
&\leq (1-\gamma)^{-1} \gamma^{-m(n_1+1)}\leq (1-\gamma)^{-1} (n_1+1)^{\beta_{0}}\lesssim \xi^{\beta_{0}}+1,
\end{align*}
where we used that $n_1\leq \xi+2$. If $\xi<0$ or $n\geq n_1+2$ then $|z-\ui n+\gamma|\geq c_{\gamma}:=\sqrt{1+\gamma^2}>1$. Therefore
\[
\|(z-\ui n+\gamma-B_{m(n)})^{-1}\|_{\La(\ell^{2}_{m(n)})} \leq \sum_{k=0}^{\infty} \frac{1}{c_{\gamma}^{k+1}}<\infty,
\]
and now \eqref{eq:wrobel} follows. In fact, \eqref{eq:wrobel} is optimal for $\eta=0$ (see \cite[Example 4.1]{Wrobel89}).

We now show that $\|T(\cdot)\|_{\calL(X_{\tau},X)}$ is unbounded for $\tau\in[0,\frac{1-\gamma}{\log(1/\delta)})$. First note that $\|T(t)\|_{\calL(X_{\tau},X)} \geq \frac{\|T(t) x\|_{X}}{\|(1+A)^{\tau} x\|_{X}}$ for each $x\in X_{\tau}$ with $1=\|x\|_{X_{\tau}}\gtrsim \|(1+A)^{\tau}x\|_{X}$. Let $n\geq n_{0}$ and let $x = (x^{(k)})_{k\geq n_{0}}\in X$ be such that $x^{(k)} = 0$ for all $k\neq m(n)$ and $x^{(m(n))} = (0, \ldots, 0, 1)$. Then, for $\tau\in \N_0$, Newton's binomial formula yields
\begin{equation}\label{eq:binomial estimate}
\|(1+A)^{\tau} x\|_{X} = \|(-\ui n+1+\gamma-B_{m(n)})^{\tau} x^{(m(n))}\|_{\ell^{2}_{m(n)}} \lesssim n^{\tau}.
\end{equation}
The moment inequality \cite[Proposition 6.6.4]{Haase06a} extends \eqref{eq:binomial estimate} to all $\tau\in[0,\infty)$. Now set $t := m(n)-1\in[1,\infty)$. Lemma \ref{lem:exp} yields
\begin{align*}
\|T(t) x\|_{X} &= \ue^{-\gamma t} \|e^{t B_{m(n)}}x^{(n)}\|_{\ell^{2}_{m(n)}} = e^{-\gamma t}  \Big(\sum_{k=0}^{m(n)-1} \Big(\frac{t^k}{k!}\Big)^2\Big)^{1/2}\\
& \gtrsim \frac{e^{(1-\gamma)m(n)}}{(m(n))^{1/4}} \gtrsim \frac{n^{\frac{1-\gamma}{\log(1/\delta)}}}{\log(n)^{1/4}}.
\end{align*}
Combining this with \eqref{eq:binomial estimate} shows that, with
$v:=\frac{1-\gamma}{\log(1/\delta)}-\tau$,
\[\|T(t)\|_{\calL(X_{\tau},X)}\gtrsim \frac{n^{v}}{\log(n)^{1/4}}\eqsim \frac{e^{tv}}{t^{1/4}}\]
for an implicit constant independent of $n\geq n_{0}$ and $t\geq1$. The latter is bounded as $n\to \infty$ if and only if $\tau\geq \frac{1-\gamma}{\log(1/\delta)}$ holds, and otherwise it grows exponentially.
\end{example}

\subsubsection{Operator matrices}

We now give an example of an operator $\mathcal{A}$ with resolvent growth $(n,0)$, for $n\in\N\setminus\{1\}$, such that $\|T(\cdot)\|_{\La(X^{m},X)}$ is unbounded for all $m\in\{0,\ldots, n-2\}$. Moreover, $\|T(t)\|_{\La(X^{n-1},X)}$ does not tend to zero as $t\to\infty$. Hence the example would show that the exponent $\frac{\sigma+1}{\alpha}-1$ in Theorem \ref{thm:asymptotically analytic semigroups} is sharp, if $A$ were a sectorial operator. However, it turns out that this is not the case. As noted in Section \ref{subsec:non sectorial operators}, our theory also applies to operators which are not sectorial.%\eqref{eq:polynomial decay Fourier type nonsectorial} is sharp for $s=\beta=0$  in the special case that shows that for every $n\in \N$, there exists an operator which has resolvent growth $(0,n)$ and $\|T(t)\|_{\La(X^{n-1},X)}$ does not decay to zero. This shows that the condition $\rho<\frac{\sigma+1}{\alpha}-1$ in Corollary \ref{cor:polynomial decay Hilbert space} (in the non-sectorial setting) is sharp. 

\begin{example}\label{ex:nonsect}
Fix $n\in\N\setminus\{1\}$. We give an example of an injective bounded operator $\mathcal{A}$ with dense range on a Hilbert space $X$ such that $\sigma(\mathcal{A}) = [0,1]$,
\begin{align}\label{eq:toproveoptimal1}
\sup\Big\{\frac{\abs{\lambda}^{n}}{(1+\abs{\lambda})^{n}}\|(\lambda+A)^{-1}\|_{\La(X)}\Big| \lambda\in\overline{\C_{+}}\setminus\{0\}\Big\}<\infty,
\end{align}
and
\begin{align}\label{eq:toproveoptimal2}
\|T(t) \mathcal{A}^{m}\|_{\La(X)}\eqsim t^{n-1-m}\qquad (t\in[1,\infty))
\end{align}
for all $m\in\{0,\ldots, n-1\}$, where $(T(t))_{t\geq0}$ is the $C_{0}$-semigroup generated by $-\mathcal{A}$. Moreover, $\mathcal{A}$ is not sectorial.

We construct $\mathcal{A}$ using operator matrices. Let $A\in \calL(L^2(0,1))$ be the multiplication operator given by $A f(x) := xf(x)$ for $f\in L^{2}(0,1)$ and $x\in(0,1)$.
Set $X:=(L^{2}(0,1))^{n}$ and let $N\in\La(X)$ be the nilpotent operator matrix with $N_{k,k+1} = I_{L^{2}(0,1)}$ for $k\in \{1, \ldots, n-1)$, and $N_{k,l}=0\in \La(L^{2}(0,1))$ for $k,l\in\{1,\ldots, n\}$ with $l\neq k+1$. Set $\mathcal{A}:=A I_{X}-N$. Then $\mathcal{A}$ is bounded and has dense range. Let $(T(t))_{t\geq0}\subseteq\La(X)$ and $(S(t))_{t\geq 0}\subseteq\La(X)$ be the $C_{0}$-semigroups generated by $-\mathcal{A}$ and $-AI_{X}$. Then $T(t) = S(t) \ue^{tN}$ for all $t\in[0,\infty)$, where we use that $AI_X$ and $N$ commute. Since $N^k\neq 0$ if and only if $k\leq n-1$, one has $\|T(t)\|_{\La(X)}\eqsim t^{n-1}$ for $t\geq0$. Also, $\sigma(\mathcal{A}) = [0,1]$ and, using the Neumann series for the resolvent,
\[
R(\lambda, \mathcal{A}) = R(\lambda, A) (I_X + R(\lambda,A) N)^{-1} = \sum_{k=0}^{n-1} R(\lambda,A)^{k+1}(-N)^k
\]
for $\lambda\in\C\setminus[0,1]$. This implies \eqref{eq:toproveoptimal1}.

Fix $m\in\{0,\ldots, n-1\}$. Then $\mathcal{A}^{m} = \sum_{k=0}^{m} \binom{m}{k} (-1)^{k} A^{m-k} N^k$ and
\begin{align}\label{eq:toproveoptimalhelp}
T(t)\mathcal{A}^{m} = \sum_{k=0}^{m} \binom{m}{k} (-1)^{k} S(t) A^{m-k} e^{tN} N^k\qquad(t\in[0,\infty)).
\end{align}
Let $k\in\{0,\ldots, m\}$ and $t\geq m$. Then
\[
\displaystyle \|S(t)\restriction_{L^{2}(0,1)} (-A)^{m-k}\|_{\La(L^{2}(0,1))} = \sup_{s\in (0,1)} e^{-ts} s^{m-k} \eqsim t^{k - m}.
\]
The dominating matrix element of $e^{tN} N^k$ is $\frac{t^{n-k-1}}{(n-k-1)!}I_{L^2(0,1)}$. Hence
\[
\|S(t) A^{n-1-k} e^{tN} N^k\|_{\La(X)} \eqsim t^{k-m} t^{n-k-1}
\]
for an implicit constant independent of $t$. Now \eqref{eq:toproveoptimal2} follows from \eqref{eq:toproveoptimalhelp}.
\end{example}

\subsubsection{Multiplication operators on Sobolev spaces}

We now consider another typical setting where one encounters generators of unbounded semigroups with polynomial growth of the resolvent. It is included to show that even straightforward multiplication operators can generate unbounded $C_0$-semigroups when the underlying space is a Sobolev space. The example also shows that Proposition \ref{prop:decay implies resolvent bounds2} is sharp.

\begin{example}\label{ex:mult}
Fix $a\in(0,\infty)$ and $b\in(0,1)$ with $a+b\geq 1$. Set $\ph(s) := s^{-a} + \ui s^b$ for $s\in(1,\infty)$. Let $X := W^{1,2}(1, \infty)$ and let $A$ be the multiplication operator on $X$ associated with $\ph$. Then $\sigma(A)\subseteq \C_{+}$ and $-A$ generates the $C_{0}$-semigroup $(T(t))_{t\geq 0}\subseteq\La(X)$ given by $T(t)f(s)=\ue^{-t\ph(s)}f(s)$ for $t\in[0,\infty)$, $f\in X$ and $s\in(1,\infty)$. We prove that $A$ has resolvent growth $(0,\frac{b-1+2a}{b})$, by showing that $\|(\eta-\ui \xi + A)^{-1}\|_{\La(X)}\lesssim g(\xi):=\abs{\xi}^{\frac{b-1+2a}{b}}$ for each $\eta\in[0,\infty)$ and $\xi\in\R$.

First note that the operator $(\eta- \ui\xi+ A)^{-1}$ is the multiplication operator on $X$ associated with $s\mapsto -(\eta+s^{-a} + \ui(s^b -\xi))^{-1}$. Furthermore,
\[
\sup\{\|(\eta-\ui \xi + A)^{-1}\|_{\La(X)}\mid \eta\in[0,\infty), \xi\in[-(a/b)^{b/(a+b)},(a/b)^{b/(a+b)}]\}<\infty,
\]
where we use that $-A$ is a semigroup generator and that $\sigma(A)\subseteq\C_{+}$. For $\xi\in\R$ with $\abs{\xi}>(a/b)^{b/(a+b)}$ we bound $\|(\eta- \ui\xi+ A)^{-1}\|_{\La(X)},$ using the supremum norm of $s\mapsto -(\eta+s^{-a} + \ui(s^b -\xi))^{-1}$ and its derivative, by
\begin{equation}\label{eq:help equation example}
\begin{aligned}
&\sup_{s\in(1,\infty)} \frac{1}{\abs{\eta+s^{-a} + \ui(s^b -\xi)}}  +\! \sup_{s\in(1,\infty)}\! \frac{|-a s^{-a-1} + \ui b s^{b-1}|}{\abs{\eta+s^{-a}+ \ui(s^b -\xi)}^{2}}\\
& \leq \sqrt{2}\!\sup_{s\in(1,\infty)} \frac{1}{s^{-a} + |s^b -\xi|}  + \sup_{s\in(1,\infty)} \frac{as^{-a-1} + b s^{b-1}}{s^{-2a} + (s^b -\xi)^2}.
\end{aligned}
\end{equation}
For the first term in \eqref{eq:help equation example} note that
\[
\sup\Big\{\frac{1}{s^{-a} + |s^b -\xi|}\Big| x\in[1,\abs{\xi}^{1/b}]\Big\}\leq \xi^{a/b}\leq g(\xi)
\]
and that $s\mapsto (s^{-a}+\abs{s^{b}-\xi})^{-1}$ is decreasing for $s>\abs{\xi}^{1/b}> (a/b)^{1/(a+b)}$. For the second term and for $\abs{\xi}>(a/b)^{b/(a+b)}>1$ and $s\in(1, (\abs{\xi}+\tfrac12)^{1/b})$, write
\[
\frac{as^{-a-1} + b s^{b-1}}{s^{-2a} + (s^b -\xi)^2}\leq a s^{a-1} + \frac{bs^{b-1}}{s^{-2a}}\lesssim g(\xi).
\]
We conclude that $A$ indeed has resolvent growth $(0,\frac{b-1+2a}{b})$.

Let $t\in[1,\infty)$ and write
\begin{align*}
&\|T(t)\|_{\calL(X)}  \eqsim \!\sup_{k\in \{0,1\}}\sup_{s\in(1,\infty)} \Big|\frac{\ud^k}{\ud s^k}[e^{-t \phi(s)}]\Big| \\
&\eqsim\! \sup_{s\in(1,\infty)} |\ui bs^{b-1} t e^{-t s^{-a}} - a s^{-a-1} t e^{-t s^{-a}}| \eqsim\! \sup_{s\in(1,\infty)} s^{b-1} t e^{-t s^{-a}} \eqsim t^{1-\frac{1-b}{a}}
\end{align*}
for implicit constants independent of $t$. It follows from Corollary \ref{cor:growth of semigroup on X} that
\[
\|T(t)\|_{\La(X_{\tau},X)} \lesssim t^{1-\frac{1-b}{a}-\rho} \qquad (t\in[1,\infty))
\]
for each $\tau\in[0,\infty)$ and $\rho\in[0,\tau b/(b-1+2a))$.
On the other hand, explicit computations yield
\begin{equation}\label{eq:help equation example 2}
\begin{aligned}
\|T(t)\|_{\calL(X_{\tau}, X)} & \eqsim \sup_{k\in \{0,1\}}\sup_{s\in(1,\infty)} \Big|\frac{\ud^k}{\ud s^k}[e^{-t \phi(s)} \phi(s)^{-\tau}]\Big|\\
&\eqsim \sup_{s\in(1,\infty)} s^{b-1-b\tau} t e^{-t s^{-a}} \eqsim t^{1-\frac{1-b+b\tau}{a}}.
\end{aligned}
\end{equation}
Thus $\|T(\cdot)\|_{\La(X_{\tau},X)}$ decays faster than Corollary \ref{cor:growth of semigroup on X} would imply. We also obtain from \eqref{eq:help equation example 2} that $\|T(t)\|_{\calL(X_{\tau}, X)}\in L^1[0,\infty)$ if and only if $\tau>\frac{b-1+2a}{b}$. Therefore Proposition \ref{prop:decay implies resolvent bounds2} yields that $A$ has resolvent growth $(0, \beta)$ for each $\beta>\frac{b-1+2a}{b}$. Since the notions of uniform boundedness and $R$-boundedness coincide on the Hilbert space $X$, this shows that the parameters in Proposition \ref{prop:decay implies resolvent bounds2} cannot be improved.
\end{example}

\section{Exponential stability}\label{sec:exponential stability}

In this section we use the theory from the previous sections to derive in a unified manner various corollaries on exponential stability.

Let $-A$ be the generator of a $C_{0}$-semigroup $(T(t))_{t\geq 0}$ on a Banach space $X$. Set $s(-A):=\sup\{\Real(\lambda)\mid\lambda\in\sigma(-A)\}$, and for $\beta\in[0,\infty)$ let
\begin{align*}
s_{\beta}(-A)&:=\inf\{\w>s(-A)\mid \sup\{(1+\abs{\lambda})^{-\beta}\|(\lambda+A)^{-1}\|_{\La(X)}\mid \Real(\lambda)\geq\w\}<\infty\},\\
s_{R}(-A)&:=\inf\{\w>s(-A)\mid R_{X}(\{(\lambda+A)^{-1}\mid\Real(\lambda)\geq\w\})<\infty\}.
\end{align*}
Then Proposition \ref{prop:decay implies resolvent bounds2} yields
\[
s_{0}(-A)\leq s_{R}(-A)\leq \w_{0}(T).
\]
In particular, for each $\eta\in (\w_{0}(T),\infty)$ the operator $A+\eta$ is sectorial. Hence for $\beta\in [0,\infty)$ the fractional domain $X_{\beta}=D((\eta+A)^{\beta})$ is defined as in Section \ref{subsec:sectorial operators}, and up to norm equivalence $X_{\beta}$ does not depend on the choice of $\eta\in(\w_{0}(T),\infty)$. Throughout this section we fix a choice of $\eta\in(\w_{0}(T),\infty)$ and the associated spaces $X_{\beta}$ for $\beta\in[0,\infty)$. For $x\in X$ let
\begin{align*}
\w(x):=\inf\{\w\in\R\mid \lim_{t\to\infty}\|e^{-\omega t}T(t)x\|_{X}=0\},
\end{align*}
and for a Banach space $Y$ continuously embedded in $X$ set
\begin{align*}
\w_{Y}(T):=\sup\{\w(x)\mid x\in Y\}.
\end{align*}
For $\beta\in(0,\infty)$ we write $\omega_\beta(T):=\omega_{X_{\beta}}(T)$. The uniform boundedness principle implies that for all $\omega>\omega_{Y}(T)$ there exists an $M\in(0,\infty)$ such that
\begin{equation}\label{eq:growthboundY}
\|T(t)\|_{\calL(Y, X)}\leq M\ue^{\omega t}\qquad (t\in[0,\infty)).
\end{equation}

We need two preparatory lemmas. The first is \cite[Lemma 3.5]{Weis97}, and it follows directly from Lemma \ref{lem:interpolation} and from basic properties of convex functions.

\begin{lemma}\label{lem:continuity growth bound}
Let $-A$ be the generator of a $C_{0}$-semigroup $(T(t))_{t\geq0}$ on a Banach space $X$. Then the function $(0,\infty)\to[-\infty,\infty)$, $\beta \mapsto \w_{\beta}(T)$, is continuous on open subintervals of $\{\beta\in\R\mid \w_{\beta}(T)\in(-\infty,\infty)\}$.
\end{lemma}

The following lemma is \cite[Theorem 3.1]{Weis97} (see also \cite[Theorem 3.2]{Weis-Wrobel96}). We show that it follows directly from Lemma \ref{lem:continuity growth bound} and Proposition \ref{prop:general semigroups}.

\begin{lemma}\label{lem:general Banach space stability}
Let $-A$ be the generator of a $C_{0}$-semigroup $(T(t))_{t\geq0}$ on a Banach space $X$. Then
$\w_{\beta+1}(T)\leq s_{\beta}(-A)$ for all $\beta\in[0,\infty)$.
\end{lemma}
\begin{proof}
First note that by Lemma \ref{lem:continuity growth bound} it suffices to show that $\w_{\beta+1+\varepsilon}(T)\leq s_{\beta}(-A)$ for all $\beta\geq0$ and $\varepsilon\in (0,1)$. Also, by a scaling argument we may suppose that $s_{\beta}(-A)<0$ and prove that $\w_{\beta+1+\varepsilon}(T)\leq0$. But in this case $A$ has resolvent growth $(0,\beta)$, and Proposition \ref{prop:general semigroups} then shows that $\sup_{t\geq0}\|T(t)\|_{\La(X_{\beta+1+\veps},X)}<\infty$.
\end{proof}

\subsection{The resolvent as an $(L^p, L^q)$ Fourier multiplier}

The following theorem is the main link between exponential stability and $(\Ellp,\Ellq)$-Fourier multipliers.  This result appeared in \cite{Hieber01} and in full generality in \cite[Theorem 3.6]{Latushkin-Shvydkoy01}. Here we give a proof using Theorem \ref{thm:abstract polynomial stability}.

\begin{theorem}\label{thm:abstract stability result}
Let  $-A$ be the generator of a $C_{0}$-semigroup $(T(t))_{t\geq0}$ on a Banach space $X$, and let $\beta\in[0,\infty)$.
Then, for all $p\in[1,\infty)$ and $q\in[p,\infty]$,
\begin{align}\label{eq:identityomegaalpha}
\w_{\beta}(T)=\inf\{\w>s_{\beta}(-A)\!\mid\! (\w+\ui\cdot+A)^{-1}\!\in\!\Ma_{p,q}(\R;\La(X_{\beta},X))\}.
\end{align}
In fact, $(\w+\ui\cdot+A)^{-1}\!\in\!\Ma_{p,q}(\R;\La(X_{\beta},X))$ for all $\w>\w_{\beta}(T)$.
\end{theorem}

\begin{proof}
Fix $p\in[1,\infty)$ and $q\in[p,\infty]$, and denote the right-hand side of \eqref{eq:identityomegaalpha} by $\mu_{p,q,\beta}(A)$.
We first show that $\w_{\beta}(T)\geq \mu_{p,q,\beta}(A)$. Let $\omega>\w_{\beta}(T)$, and apply Proposition \ref{prop:decay implies resolvent bounds2} to $(\ue^{-\w t}T(t))_{t\geq 0}$ to obtain $\omega>s_\beta(-A)$.  Now $(\w+\ui\cdot+A)^{-1}\in\!\Ma_{p,q}(\R;\La(X_{\beta},X))$ follows from Lemma \ref{lem:standardfact} and Young's inequality for convolutions. Hence $\omega\geq \mu_{p,q,\beta}(A)$, and the statement follows by letting $\omega\downarrow \w_{\beta}(T)$.

To prove the reverse inequality it suffices to assume that $\mu_{p,q,\beta}(A)<0$ and show that $\w_{\beta}(A)\leq 0$. Note that $R(i\cdot,A)\in \Ma_{p,q,\beta}(\R;\La(X_{\beta},X))$. Indeed, this follows by using Proposition \ref{prop:app-equivalence growth bounds} and \cite[Theorem 5.18]{Rosenblum-Rovnyak} to express $(\ui\cdot+A)^{-1}\in L^{\infty}(\R;\La(X_{\beta},X))$ as a convolution of the Poisson kernel with $(\w+\ui\cdot+A)^{-1}$ for $s_{\beta}(-A)<\w<0$, and by applying Young's inequality. From Theorem \ref{thm:abstract polynomial stability} with $\psi \equiv 0$ one now obtains $\sup_{t\geq0}\|T(t)\|_{\calL(X_{\beta},X)}<\infty$ and $\omega_\beta(T)\leq 0$. Here one may use Lemma \ref{lem:general Banach space stability} to see that $X_{\beta}$ satisfies the assumptions of Theorem \ref{thm:abstract polynomial stability}.
\end{proof}

The first part of the following theorem is \cite[Theorem 3.2]{Weis97} (see also \cite[Theorem 4.4]{vNeStWe95} and \cite[Remark 3.3]{Weis-Wrobel96}). The proof avoids the use of Mikhlin's multiplier theorem on Besov spaces (see \cite[Theorem 2.1]{Weis97}) and instead relies on the elementary Proposition \ref{prop:Lp-Lq multipliers Fourier type}. Part \eqref{item:exp2} is the main result of \cite{vanNeerven09}.

\begin{theorem}\label{thm:exponential stability}
Let $-A$ be the generator of a $C_{0}$-semigroup $(T(t))_{t\geq0}$ on a Banach space $X$. Then the following assertions hold:
\begin{enumerate}
\item\label{item:exp1} If $X$ has Fourier type $p\in[1,2]$ then $\w_{\frac{1}{p}-\frac{1}{p'}}(T)\leq s_{0}(-A)$.
\item\label{item:exp2} If $X$ has type $p\in[1,2]$ and cotype $q\in[2,\infty]$ then $\w_{\frac{1}{p}-\frac{1}{q}}(T)\leq s_{R}(-A)$.
\end{enumerate}
\end{theorem}
\begin{proof}
By Lemma \ref{lem:continuity growth bound} and a scaling argument, for \eqref{item:exp1} we may assume that $s_{0}(-A)<0$ and show that $\w_{\frac{1}{p}-\frac{1}{p'}+\veps}(T)\leq 0$ for any $\veps>0$. The latter follows directly from Theorem \ref{thm:polynomial decay Fourier type}. In the same way \eqref{item:exp2} follows from Theorem \ref{thm:polynomial decay type}. Alternatively, one can give direct proofs by combining Theorem \ref{thm:abstract stability result} with Proposition \ref{prop:app-equivalence growth bounds} and the multiplier results in Propositions \ref{prop:Lp-Lq multipliers Fourier type} and \ref{prop:Lp-Lq multipliers type}.
\end{proof}

The geometry of $X$ and regularity of the resolvent can be used to obtain $R$-bounds from uniform bounds, leading to the following corollary.

\begin{corollary}\label{cor:R-boundedness for free}
Let $-A$ be the generator of a $C_{0}$-semigroup on a Banach space $X$ with type $p\in[1,2]$ and cotype $q\in[2,\infty]$.  Then $\w_{\frac{2}{p}-\frac{2}{q}}(T)\leq s_{0}(-A)$.
\end{corollary}
\begin{proof}
Let $\frac1r = \frac1p-\frac1q$ and $\beta>\frac1p-\frac1q$. By Lemma \ref{lem:general Banach space stability} we may suppose that $\beta\in(0,1/2)$. By Lemma \ref{lem:continuity growth bound}, Theorem \ref{thm:abstract stability result} and Proposition \ref{prop:Lp-Lq multipliers type} it suffices to show that for each $\omega>s_0(A)$ the set $\{(1+|\xi|)^\beta (\omega+\ui\xi+A)^{-1}\mid \xi\in\R\}\subseteq\La(X_{2\beta},X)$ is $R$-bounded. Let $E := \calL(X_{2\beta}, X)$ and define $f:\R\to E$ by $f(\xi) := (1+|\xi|)^\beta (\omega+\ui\xi+A)^{-1}$ for $\xi\in\R$. Then $\|f(\xi)\|_{E}\leq C (1+|\xi|)^{-\beta}$ by Proposition \ref{prop:app-equivalence growth bounds}, so that $f\in L^r(\R;E)$. Moreover,
\begin{align*}
\|f'(\xi)\|_{\calL(E)}&\leq \beta (1+|\xi|)^{\beta-1}\|(\omega+\ui \xi+ A)^{-1}\|_{E} +  (1+|\xi|)^{\beta} \|(\omega+\ui\xi+A)^{-2}\|_{E}
\\ & \lesssim  (1+|\xi|)^{-\beta-1} + (1+|\xi|)^{-\beta}.
\end{align*}
So $f\in W^{1,r}(\R;E)$, and Lemma \ref{lem:rbddsuff} shows that the range of $f$ is $R$-bounded.
\end{proof}

For $X = L^r(\Omega)$ with $r\in [1, \infty)$, Corollary \ref{cor:R-boundedness for free} and part \eqref{item:exp1} of Theorem \ref{thm:exponential stability} yield the same conclusion. It is an open question whether in this case the index $|\frac{2}{r}-1|$ can be improved.

\begin{remark}\label{rem:extra parameter}
In Theorem \ref{thm:exponential stability} and Corollary \ref{cor:R-boundedness for free} one can add a parameter $\beta\in[0,\infty)$, as in Lemma \ref{lem:general Banach space stability}. Then Theorem \ref{thm:exponential stability} \eqref{item:exp1} says that $\w_{\beta+\frac{1}{p}-\frac{1}{p'}}(T)\leq s_{\beta}(-A)$, and \eqref{item:exp2} that $\w_{\beta+\frac{1}{p}-\frac{1}{q}}(T)\leq s_{R,\beta}(-A)$. Here 
\[
s_{R,\beta}(-A):=\inf\{\w>s(-A)\mid R_{X}(\{(1+\abs{\lambda})^{-\beta}(\lambda+A)^{-1}\mid \Real(\lambda)\geq\w\})<\infty\}.
\]
In Corollary \ref{cor:R-boundedness for free} the more general inequality is $\w_{2\beta+\frac{2}{p}-\frac{2}{p'}}(T)\leq s_{\beta}(-A)$. The proofs are the same, using Proposition \ref{prop:app-equivalence growth bounds}.
\end{remark}

\subsection{The resolvent as a Fourier multiplier on Besov spaces}\label{subsec:multipliers on Besov spaces}

In this subsection we give an alternative characterization of $\omega_{\beta}(T)$, $\beta>0$, using Fourier multipliers on Besov spaces. We then use this characterization to obtain a new stability result for positive semigroups.

For the definition and basic properties of vector-valued Sobolev and Besov spaces which are used below we refer to \cite{ScScSi12,Schmeisser-Sickel04}.

\begin{theorem}\label{abstract stability result using Bessel and Besov spaces}
Let $-A$ be the generator of a $C_{0}$-semigroup $(T(t))_{t\geq0}$ on a Banach space $X$, and let $\beta\in(0,\infty)$. Then for all $p\in[1,\infty)$ and $q\in[p,\infty]$,
\[
\w_{\beta}(T) = \inf\{\w>s_{\beta}(-A)\!\mid\! T_{m_{\omega}}\in \calL(\Be^{\beta}_{p,1}(\R;X), L^q(\R;X))\},
\]
where $m_{\omega}(\cdot):= (\w+\ui\cdot+A)^{-1}$ for $\w>s_{\beta}(-A)$.
\end{theorem}
\begin{proof}
Denote the right-hand side of \eqref{eq:identityomegaalpha} by $\nu_{p,q,\beta}(A)$. We first show that $\w_\beta(T) \leq \nu_{p,q,\beta}(A)$.
By shifting $A$ and using Lemma \ref{lem:continuity growth bound} we may assume that $\nu_{p,q,\beta}(A)<0$ and prove that $\w_{\beta+\veps}(T)\leq 0$ for any $\varepsilon>0$.  Without loss of generality we may suppose that $(\beta,\beta+\varepsilon]\cap \N = \varnothing$. Let $n := \lceil \beta+\varepsilon\rceil$, $\alpha\in (\beta, \beta+\varepsilon)$ and let $\D_{A}(\alpha,1)=(X,\D(A^{n}))_{\alpha/n,1}$ be the appropriate real interpolation space.
Set $m(\xi):= (\ui\xi+A)^{-1}$ for $\xi\in\R$. As in Theorem \ref{thm:abstract stability result} one sees that $T_{m}\in\La(\Be^{\alpha}_{p,1}(\R;X),\Ellq(\R;X))$, where we also use that $B^{\alpha}_{p,1}(\R;X)\subseteq B^{\beta}_{p,1}(\R;X)$ continuously.

Let $\w>\w_{0}(T)$ and let $F\in\La(X,L^p(\R;X))$ be given by $Fx(t) := t^n \ue^{-\w t}T(t)x$ for $t\in\R$ and $x\in X$, where we extend the semigroup by zero to all of $\R$. Then $F:X_n\to W^{n,p}(\R;X)$ is bounded and, by real interpolation, 
\[
F:\D_{A}(\alpha,1)\to (\Ellp(\R;X),\W^{n,p}(\R;X))_{\alpha/n,1} = \Be^{\alpha}_{p,1}(\R;X)
\]
is bounded. Now fix $x\in X_{\alpha+1}$ and let $f:= Fx$. Then
\begin{equation}\label{Lp-Lq-estimate}
\|T_{m}(f)\|_{\Ellq(\R;X)}\lesssim\|f\|_{\Be^{\alpha}_{p,1}(\R;X)}\lesssim\|x\|_{\D_{A}(\alpha,1)}\lesssim \|x\|_{X_{\beta+\varepsilon}},
\end{equation}
where we have also used that $X_{\beta+\varepsilon} \hookrightarrow X_{\D_{A}(\alpha,1)}$. By Lemmas \ref{lem:standardfact} and \ref{lem:general Banach space stability} one has
\begin{equation}\label{eq:multiplier is convolution}
T\ast f(t):=\int_{0}^{t}T(t-s)f(s)\ud s = T_{m}(f)(t)\qquad(t\in[0,\infty))
\end{equation}
and $T_{m}(f)(t)=0$ for $t\in(-\infty,0)$. On the other hand, $T(t-s)f(s) = s^n \ue^{-\w s}T(t) x$ for $s\in[0,t]$. Hence, for $t\geq 1$,
\[
\|T(t)x\|_{X}\lesssim \|T(t)x\|_{X}  \int_0^t s^n e^{-\w s}\ud s=\|T*f(t)\|_{X}.
\]
Now, using that $\sup_{t\in[0,1]}\|T(t)\|_{\La(X)}<\infty$, it follows from \eqref{Lp-Lq-estimate} and \eqref{eq:multiplier is convolution} that
\[
\Big(\int_0^\infty \|T(t)x\|_{X}^q \ud t\Big)^{\frac1q}\lesssim \|x\|_{X}+ \|T*f\|_{\Ellq([1,\infty);X)}\lesssim\|x\|_{X_{\beta+\varepsilon}}.
\]
Therefore Proposition \ref{prop:abstract polynomial stability} implies that $R(i\cdot,A)\in \Ma_{1,q}(\R;\La(X_{\beta+\varepsilon},X))$. Here one may again use Lemma \ref{lem:general Banach space stability} to see that $X_{\beta+\veps}$ satisfies the conditions of Proposition \ref{prop:abstract polynomial stability}. Finally, Theorem \ref{thm:abstract stability result} shows that $\omega_{\beta+\varepsilon}(A)\leq 0$.

Next, we prove that $\w_\beta(T) \geq \nu_{p,q,\beta}(A)$. To do so it suffices by Lemma \ref{lem:continuity growth bound} to show that for all $\alpha\in(0,\beta)$ and $\omega>\w_{\alpha}(T)$ one has
$T_{m_{\omega}}\in \calL(\Be^{\beta}_{p,1}(\R;X), L^q(\R;X))$. Moreover, we may suppose that $\alpha\notin \N$. Let $n\in \N$ be such that $\alpha, \beta\in (n-1, n]$. By \eqref{eq:growthboundY} there exist $M,\varepsilon\in(0,\infty)$ such that $\ue^{-\w t}\|T(t)\|_{\La(X_{\alpha},X)}\leq M e^{-\varepsilon t}$ for all $t\geq 0$. For $f\in\Sw(\R;X)$ set $S_{\alpha} f := (\w+ A)^{-\alpha}T_{m_\omega}f$. Then $S_{\alpha}\in\La(L^p(\R;X),L^q(\R;X))$ by Theorem \ref{thm:abstract stability result}. We claim that $S_{\alpha}\in\La(W^{k,p}(\R;X),L^q(\R;X_k))$ for each $k\in[1,\ldots, n\}$. Indeed, let $f\in C^k_c(\R)\otimes X$ and $t\in[0,\infty)$. Lemma \ref{lem:standardfact} and integration by parts yield
\begin{align*}
&(\w+A)^k S_{\alpha} f(t)  = \int_{-\infty}^t (\w+A)^k e^{-\w(t-s)}T(t-s) (\w+A)^{-\alpha}f(s) \ud s
\\ &  = -\int_{-\infty}^t \frac{d}{ds} [(\w+A)^{k-1} e^{-\w(t-s)}T(t-s)] (\w+A)^{-\alpha}f(s) \ud s
\\ &  = -(\w+A)^{k-1-\alpha} f(t) + \int_{-\infty}^t (\w+A)^{k-1} e^{\w(t-s)}T(t-s) (\w+ A)^{-\alpha}f'(s) \ud s
\\ & = -(\w+A)^{k-1-\alpha} f(t) + (\w+A)^{k-1} S_{\alpha} f'(t).
\end{align*}
By iterating this procedure one obtains
\[
(\w+A)^k S_{\alpha} f(t) = -\sum_{j=1}^{k}   (\w+A)^{k-j-\alpha} f^{(j-1)}(t)  + S_{\alpha} f^{(k)}(t).
\]
Since $k-1-\alpha<0$ this yields
\begin{align*}
\|(\w+A)^k S_{\alpha} f\|_{L^q(\R;X)} & \leq \sum_{j=1}^{k} \|(\w+A)^{k-j-\alpha} f^{(j-1)}\|_{L^q(\R;X)} +  \|S_{\alpha} f^{(k)}\|_{L^q(\R;X)}
\\ & \lesssim \|f\|_{W^{k-1,q}(\R;X)} + \|f^{(k)}\|_{L^p(\R;X)}\lesssim  \|f\|_{W^{k,p}(\R;X)},
\end{align*}
and the claim follows since $(\w+A)^{-k}:X\to X_{k}$ is an isomorphism.

Now, if $\beta=n$ then $B^{\beta}_{p,1}(\R;X)\subseteq W_{n,p}(\R;X)$ continuously so $S_{\alpha}:\Be^{\beta}_{p,1}(\R;X)\to L^q(\R;X_{\beta})$ is bounded. On the other hand, if $\beta<n$ then real interpolation for the exponents $k=n-1$ and $k = n$ shows that
$S_{\alpha}\in\La(\Be^{\beta}_{p,1}(\R;X),L^q(\R;\D_{A}(\beta,1)))$. Since $\alpha<\beta$, in both cases we obtain that $T_{m_\omega}:\Be^{\beta}_{p,1}(\R;X)\to L^q(\R;X)$ is bounded, which concludes the proof.
\end{proof}

The following theorem unifies \cite[Theorem 1]{Weis95} and \cite[Corollary 1.3]{vanNeerven09} and is new for $1\leq p<q<2$ and $2<p<q<\infty$. Here there is no use in adding an additional parameter $\beta$ as in Lemma \ref{lem:general Banach space stability} and Remark \ref{rem:extra parameter}, since $s_{0}(-A)=s(-A)$.

\begin{theorem}\label{thm:exponential stability convex}
Let $-A$ be the generator of a positive $C_{0}$-semigroup $(T(t))_{t\geq 0}$ on a Banach lattice $X$ which is $p$-convex and $q$-concave for $p\in[1,\infty)$ and $q\in[p,\infty)$. Then $\w_{\frac{1}{p}-\frac{1}{q}}(T)\leq s(-A)$.
\end{theorem}
\begin{proof}
First note that $s_{0}(-A)=s(-A)$ (see \cite[Theorem 5.3.1]{ArBaHiNe11}). Let $r\in (1, \infty]$ satisfy $\frac1r = \frac{1}{p} - \frac{1}{q}$ and let $\omega>s_{0}(-A)$. By Theorem \ref{abstract stability result using Bessel and Besov spaces} it suffices to show that $T_{m}\in \La(B^{1/r}_{p,1}(\R;X),L^{q}(\R;X))$ for $m(\xi) := (\omega+\ui\xi+A)^{-1}$, $\xi\in\R$. For $n\in\N$ with $n>\w_{0}(T)$ set $K_{n}(t) := e^{-\w t}T(t) n(n+A)^{-1}$, $t\geq 0$, and let $K_{n}\equiv 0$ on $(-\infty,0)$. Then $K_{n}(t)\in\La(X)$ is positive for all $t\in\R$, and $K_n(\cdot)x\in L^1(\R;X)$ for all $x\in X$ by Lemma \ref{lem:general Banach space stability}. Furthermore, $\F(K_{n}x)(\xi) = m_n(\xi)x$ for $\xi\in\R$, where $m_n(\xi) := (\omega+\ui\xi+A)^{-1}n(n+A)^{-1}$. Note that $\sup_{\xi\in\R}  \|m_n(\xi)\|_{\La(X)}<\infty$. Now, since $X$ has cotype $q<\infty$ (see \cite[p.\ 332]{DiJaTo95}), it follows from Proposition \ref{prop:positive kernel Lp-Lq multipliers} and the continuous embedding $B^{1/r}_{p,1}(\R;X)\subseteq H^{1/r}_{p}(\R;X)$ that
\[
C:=\sup_{n}\|T_{m_n}\|_{\calL(B^{1/r}_{p,1}(\R;X),L^{q}(\R;X))} <\infty.
\]
Now fix $f\in \Sw(\R;X)$. Then $\|T_{m_n}(f)\|_{L^{q}(\R;Y)} \leq C\|f\|_{B^{1/r}_{p,1}(\R;X)}$ for all $n$. Moreover, for $\xi\in\R$ and $x\in X$ one has $\lim_{n\to \infty} m_n(\xi) x =m(\xi)x$ by \cite[Lemma 3.4]{Engel-Nagel00}. Now \cite[Lemma 3.1]{Rozendaal-Veraar17a} implies that $T_{m}\in\La(B^{1/r}_{p,1}(\R;X),L^{q}(\R;X))$, as required.
\end{proof}

Using Theorem \ref{thm:exponential stability convex} and \cite[Example 5.5b]{Girardi-Weis03b} one can modify an example due to Arendt (see \cite[Example 5.1.11]{ArBaHiNe11}, \cite[Section 4]{Weis-Wrobel96} and \cite[Example 1.4]{vanNeerven09}) to construct for all $1\leq p\leq q<\infty$ a positive $C_{0}$-semigroup $(T(t))_{t\geq0}$ on $X=L^p(1, \infty)\cap L^q(1, \infty)$ with generator $-A$ such that $\w_{0}(T)=-\frac{1}{p}$, 
\[
\omega_{\frac1p - \frac1q}(T) = s_{R}(-A)=s_0(-A)=-\frac{1}{q},
\]
and such that $\alpha\mapsto \w_{\alpha}(T)$ is linear on $[0,\frac{1}{p}-\frac{1}{q}]$. This shows that the index $\frac{1}{p}-\frac{1}{q}$ in part \eqref{item:exp2} of Theorem \ref{thm:exponential stability} and in Theorem \ref{thm:exponential stability convex} is optimal, which shows in turn that \cite[Theorem 3.24]{Rozendaal-Veraar17a} is optimal. Moreover, it follows from \cite[Example 4.4]{Weis97} that the positivity assumption in Theorem \ref{thm:exponential stability convex} cannot be omitted.

\addtocontents{toc}{\protect\setcounter{tocdepth}{0}}
\subsection*{Acknowledgements}

Part of this research was conducted during a stay of the first author at the Institut Mittag-Leffler for the research program 'Interactions between Partial Differential Equations \& Functional Inequalities'. He would like to thank the organizers of the program, and in particular Bogus\l{}aw Zegarli\'n{}ski, for the invitation to the Institut Mittag-Leffler.

The authors would like to thank Fokko van de Bult for helpful discussions which led to Lemma \ref{lem:exp}, and Yuri Tomilov, Lassi Paunonen and Reinhard Stahn for useful comments. We would also like to thank the referee for carefully reading the manuscript.

\addtocontents{toc}{\protect\setcounter{tocdepth}{2}}

\appendix
\addtocontents{toc}{\protect\setcounter{tocdepth}{1}}

\section{Technical estimates}\label{sec:technical lemmas}

In this section we provide the proofs of a few technical results which are used in the main text.

\subsection{Contour integrals}

We start with a lemma which is needed when dealing with certain contour integrals in Proposition \ref{prop:app-equivalence growth bounds}.

\begin{lemma}\label{lem:app-auxiliary lemma}
Let $\ph\in(0,\tfrac{\pi}{2}]$ and $\theta\in(\pi-\ph,\pi)$. Set $\Omega:=\overline{\C_{+}}\setminus (S_{\ph}\cup\{0\})$ and let $\Gamma:=\{r\ue^{\ui\theta}\mid r\in[0,\infty)\}\cup\{r\ue^{-\ui\theta}\mid r\in[0,\infty)\}$ be oriented from $\infty\ue^{\ui\theta}$ to $\infty\ue^{-\ui\theta}$. Then for all $\alpha\in[0,\infty)$, $\beta\in(0,\infty)$, $\eta\in(0,1]$ and $\lambda\in \Omega$ one has
\begin{equation}\label{eq:app-auxiliary equation 1}
\frac{1}{2\pi\ui}\int_{\Gamma}\frac{z^{\alpha}}{(\eta+z)^{\alpha+\beta}(z+\lambda+\eta-1)}\ud z=\frac{(1-\eta-\lambda)^{\alpha}}{(1-\lambda)^{\alpha+\beta}}.
\end{equation}
Furthermore, for all $\gamma\in[1,\infty)$ and $\delta\in[0,\infty)$ there exists a constant $C\in[0,\infty)$ such that
\begin{equation}\label{eq:app-auxiliary equation 2}
\frac{\abs{z}^{\gamma}}{\abs{1+z}^{\gamma+\delta}}\frac{\abs{1-\lambda}}{\abs{z+\lambda}}\leq C\quad \text{and}\quad \frac{1+\abs{\lambda}}{\abs{\frac{1}{2}+z}^{\delta}\abs{z+\lambda-\frac{1}{2}}}\leq C
\end{equation}
for all $z\in\Gamma$ and $\lambda\in\Omega$.
\end{lemma}
\begin{proof}

Let $\alpha\in[0,\infty)$, $\beta\in(0,\infty)$, $\eta\in(0,1]$ and $\lambda\in\Omega$. For $r\in (0,\Imag(\lambda)/2]$ and $R\geq 2\abs{\lambda}+2$ set $\Gamma_{+}:=\{s\ue^{\ui\theta}\mid s\in[r,R]\}$, $\Gamma_{-}:=\{s\ue^{-\ui\theta}\mid s\in[r,R]\}$, $\Gamma_{r}:=\{r \ue^{\ui\nu}\mid \nu\in[-\theta,\theta]\}$ and $\Gamma_{R}:=\{R\ue^{\ui\nu}\mid \nu\in[-\theta,\theta]\}$, and let $\Gamma_{r,R}:=\Gamma_{+}\cup\Gamma_{r}\cup\Gamma_{-}\cup\Gamma_{R}$ be oriented counterclockwise. Then
\begin{align*}
\int_{\Gamma_{R}}\frac{\abs{z}^{\alpha}}{\abs{\eta+z}^{\alpha+\beta}\abs{z+\lambda+\eta-1}}\ud \abs{z}&=\int_{-\theta}^{\theta}\frac{R^{1+\alpha}}{\abs{\eta+R\ue^{\ui\nu}}^{\alpha+\beta}\abs{R\ue^{\ui\nu}+\lambda+\eta-1}}\ud \nu\\
&=R^{-\beta}\int_{-\theta}^{\theta}\frac{1}{\abs{\frac{\eta}{R}+\ue^{\ui\nu}}^{\alpha+\beta}\abs{\ue^{\ui\nu}+\frac{\lambda+\eta-1}{R}}}\ud \nu\\
&\leq 2^{2+\alpha+\beta}\theta R^{-\beta},
\end{align*}
and the latter tends to zero as $R\to \infty$. Similarly, one sees that
\[
\int_{\Gamma_{r}}\frac{\abs{z}^{\alpha}}{\abs{\eta+z}^{\alpha+\beta}\abs{z+\lambda+\eta-1}}\ud \abs{z}
\]
tends to zero as $r\to 0$. Now Cauchy's integral theorem yields
\begin{align*}
&\frac{1}{2\pi\ui}\int_{\Gamma}\frac{z^{\alpha}}{(\eta+z)^{\alpha+\beta}(z+\lambda+\eta-1)}\ud z\\
&=\lim_{r\to0, R\to\infty}\frac{1}{2\pi\ui}\int_{\Gamma_{r,R}}\frac{z^{\alpha}}{(\eta+z)^{\alpha+\beta}(z+\lambda+\eta-1)}\ud z=\frac{(1-\eta-\lambda)^{\alpha}}{(1-\lambda)^{\alpha+\beta}},
\end{align*}
which proves \eqref{eq:app-auxiliary equation 1}.

Next, let $\gamma\in[1,\infty)$, $\delta\in(0,\infty)$, $z\in \Gamma$ and $\lambda\in\Omega$. Note that $\abs{z+\lambda}=\abs{z}\,\abs{\ue^{\pm\ui\theta}+\lambda'}$ for some $\lambda'\in\Omega$. Since the distance from $\ue^{\pm\ui\theta}$ to $-\Omega$ is nonzero, there exists a constant $C_{1}\in(0,\infty)$ such that $\abs{z+\lambda}\geq C_{1}\abs{z}$. Hence
\begin{align*}
\frac{\abs{z}^{\gamma}}{\abs{1+z}^{\gamma+\delta}}\frac{\abs{1-\lambda}}{\abs{z+\lambda}}&\leq \frac{\abs{z}^{\gamma}}{\abs{1+z}^{\gamma+\delta}}\Big(\frac{\abs{1+z}}{\abs{z+\lambda}}+1\Big)
 \leq \frac{\abs{z}^{\gamma}}{\abs{1+z}^{\gamma+\delta}}\Big(\frac{C_{1}^{-1}}{\abs{z}}+C_{1}^{-1}+1\Big),
\end{align*}
and the latter is uniformly bounded in $z\in \Gamma$.

For the second term in \eqref{eq:app-auxiliary equation 2} first note that the distances from $z-\frac{1}{2}$ to $\Gamma$, and hence to $-\Omega$, and from $z+\frac{1}{2}$ to $0$ are bounded uniformly from below by a constant $C_{2}>0$. Hence $\abs{z+\lambda-\frac{1}{2}}\geq C_{2}$ and $\abs{\frac12 + z}\geq C_{2}$ for all $z\in\Gamma$ and $\lambda\in\Omega$. Therefore, for the second term in \eqref{eq:app-auxiliary equation 2} it suffices to bound $\frac{\abs{\lambda}}{\abs{z+\lambda-\frac{1}{2}}}$ uniformly. Let $\nu\in[\frac{\pi}{2},\frac{\pi}{2}]$ be such that $\lambda = |\lambda| e^{i\nu}$, and set $w := \frac{z}{\abs{\lambda}} \in \Gamma$. Then
\[
\frac{\abs{\lambda}}{\abs{z+\lambda-\frac{1}{2}}} = \frac{1}{\abs{w+e^{i\nu}-\frac{1}{2|\lambda|}}}.
\]
Now the required results follows, since by geometric inspection one sees that
\[
\abs{(w-\tfrac{1}{2|\lambda|})-(-e^{i\nu})}\geq \rm{dist}(\Gamma, -e^{i\nu})\geq \rm{dist}(\Gamma, -e^{i\ph}).\qedhere
\]
\end{proof}

\subsection{Estimates for exponential functions}

The following lemma provides a two-sided exponential estimate, one part of which is used in Example \ref{ex:optimalitybeta}.

\begin{lemma}\label{lem:exp}
Let $m\in\N$. Then
\begin{equation}\label{eq:exp lemma}
\frac{\ue^{m}}{m^{1/4} \ue^2} \leq \Big(\sum_{j=0}^m \Big(\frac{m^j}{j!}\Big)^2\Big)^{1/2} \leq \frac{\ue^{m}}{m^{1/4}}.
\end{equation}
\end{lemma}
\begin{proof}
Both estimates are clear for $m=1$, so we may consider $m\geq 2$ throughout. Let $k\in\N\cap [1,\lfloor \sqrt{m}\rfloor]$ and note that
\begin{align*}
\log\Big(\frac{m^{m-k}}{(m-k)!}\Big) & = (m-k)\log(m) - \log((m-k)!)
\\ & \geq (m-k)\log(m) - (m-k+\tfrac12) \log(m-k)  + m-k - 1,
\end{align*}
where we used Stirling's formula. Moreover,
\[
\log(m-k) = \log(m) + \log(1-\tfrac{k}{m}) \leq \log(m) - \tfrac{k}{m},
\]
where we used that $\log(1-s)\leq -s$ for $s\in (0,1)$. Hence
\begin{align*}
\log\Big(\frac{m^{m-k}}{(m-k)!}\Big) &\geq (m-k)\log(m) - (m-k+\tfrac12)(\log(m) - \tfrac{k}{m}) + m-k-1
\\ &  \geq -\frac12\log(m) -\frac{k^2}{m} + m -1\geq -\frac12\log(m)  +m -2,
\end{align*}
where in the last step we used that $k^2\leq m$ holds.
We now see that $\frac{m^{m-k}}{(m-k)!}\geq m^{-1/2} \ue^{-2} \ue^{m}$, from which we deduce the first inequality in \eqref{eq:exp lemma}:
\[
\sum_{j=0}^m \Big(\frac{m^j}{j!}\Big)^2 \geq \sum_{k=0}^{\lfloor \sqrt{m}\rfloor} \Big(\frac{m^{m-k}}{(m-k)!}\Big)^2 \geq m^{\frac12} m^{-1} \ue^{-4} \ue^{2m} = m^{-1/2}\ue^{-4}  \ue^{2m}.
\]

For the second inequality let $a_j := \frac{m^j}{j!}$ for $j\in \{0, \ldots, m\}$. Then another application of Stirling's formula yields
\[
a_{m-1}\leq  \frac{m^{m-1}}{\sqrt{2\pi} (m-1)^{m-\frac12} e^{-(m-1)}} = \frac{e^{m}}{e\sqrt{2\pi}\sqrt{m}} \Big(1+\frac{1}{m-1}\Big)^{m-\frac12}\leq  \frac{e^m}{\sqrt{m}}.
\]
Also, $a_{j}\leq a_{m-1}$ for each $j\in\{0,\ldots, m-1\}$, where we use that $a_{j+1}/a_j \geq 1$ for each $j\in\{0,\ldots, m-1\}$ and that $\frac{a_{m}}{a_{m-1}}=1$. Since $\sum_{j=0}^m a_j \leq e^{m}$, the upper estimate in \eqref{eq:exp lemma} follows from
\[
\sum_{j=0}^m a_j^2 \leq a_{m-1}\sum_{j=0}^m a_j \leq \frac{e^{2m}}{\sqrt{m}}.\qedhere
\]
\end{proof}

\bibliographystyle{plain}
\bibliography{BibliografieStab}

\end{document}